\documentclass[preprint,12pt]{elsarticle}
\usepackage{amsmath,amsthm,amssymb}
\usepackage{epstopdf}
\usepackage{makecell}
\usepackage{booktabs}
\usepackage{multirow}
\usepackage{diagbox}
\usepackage{extarrows}
\biboptions{numbers,sort&compress}
\usepackage{ulem}
\usepackage[colorlinks,linkcolor=blue,anchorcolor=blue,citecolor=blue]{hyperref}
\usepackage{cleveref}
\usepackage{soul}
\usepackage{geometry}
\allowdisplaybreaks[4]
\usepackage{bm,bbm}

\newtheorem{lemma}{Lemma}
\newtheorem{theorem}{Theorem}

\newtheorem{remark}{{Remark}}
\newtheorem{example}{Example}
\numberwithin{theorem}{section} 
\numberwithin{equation}{section}
\numberwithin{lemma}{section}
\numberwithin{example}{section}
\numberwithin{definition}{section}
\numberwithin{remark}{section}

\def\dfrac{\displaystyle\frac}

\allowdisplaybreaks[4]

\journal{Journal}
\begin{document}
\begin{frontmatter}

\title{Optimal preconditioning techniques for finite volume approximation of three-dimensional conservative space-fractional diffusion equations}






\author[add1]{Wei Qu}\ead{quwei@sgu.edu.cn}
\author[add2]{Siu-Long Lei}\ead{sllei@um.edu.mo}
\author[add3]{Sean Y. Hon}\ead{seanyshon@hkbu.edu.hk}
\author[add3]{Yuan-Yuan Huang\corref{cor1}}\ead{doubleyuans@hkbu.edu.hk}

\cortext[cor1]{Corresponding author.}

\address[add1]{School of Mathematics and Statistics, Shaoguan University, Shaoguan, China.}

\address[add2]{Department of Mathematics, University of Macau, Macao SAR, China}

\address[add3]{Department of Mathematics, Hong Kong Baptist University, Hong Kong SAR, China.}

\begin{abstract}
A Crank–Nicolson finite volume approximation for three-dimensional conservative space-fractional diffusion equation results in large and dense three-level Toeplitz discrete linear systems. Preconditioned Krylov subspace methods with sine transform-based preconditioners are developed to solve these systems, including the preconditioned conjugate gradient (PCG) method for the symmetric case and the preconditioned generalized minimal residual (PGMRES) method for the non-symmetric case. Moreover, we provide detailed analysis of the convergence of these Krylov subspace methods. Specifically, for the symmetric case, we prove the spectra of the preconditioned matrices are uniformly bounded in the open interval $(1/2, 3/2)$, which results in a linear convergence rate of the PCG method. For the non-symmetric case, we demonstrate that the PGMRES method also achieves a linear convergence rate independent of discretization stepsizes from the residual point of view. These results imply that the iteration counts of the PCG and PGMRES methods are uniformly bounded and independent of the matrix sizes. Numerical experiments in both symmetric and non-symmetric cases in two- and three-dimensions are conducted to confirm the optimal performance of the proposed preconditioned Krylov subspace methods.
\end{abstract}

\begin{keyword}
Crank–Nicolson finite volume approximation, Three-level Toeplitz discrete linear systems, Sine transform-based preconditioners, Preconditioned Krylov subspace methods, linear convergence rate

\noindent{\it Mathematics Subject Classification:} 65N08, 65F08, 65F10, 15B05 

\end{keyword}

\end{frontmatter}

\section{Introduction}
Fractional diffusion equations (FDEs), as generalizations of integer-order diffusion equations, have been proven to be effective models for anomalous diffusion \cite{BWM1,BWM2,CLZ1}, biology \cite{Magin1}, and finance \cite{raberto2002waiting}. Conventionally, spatial FDEs are formed by replacing second-order derivatives in classic diffusion equations with fractional Riemann-Liouville (R-L) or Riesz derivatives. 

In this paper, we consider the solution of the following three-dimensional (3D) conservative space-FDEs (SFDEs) with order $2-\alpha$ in $x$-direction, $2-\beta$ in $y$-direction and $2-\gamma$ in $z$-direction:
\begin{equation}\label{RLFDEs}
\begin{cases}\begin{array}{ll}
\frac{\partial u({\boldsymbol{x}},t)}{\partial t}-\frac{\partial}{\partial x} \mathcal{D}_{x}^\alpha u({\boldsymbol{x}}, t)-\frac{\partial}{\partial y} \mathcal{D}_{y}^\beta u({\boldsymbol{x}}, t)-\frac{\partial}{\partial z} \mathcal{D}_{z}^\gamma u({\boldsymbol{x}}, t)=f({\boldsymbol{x}}, t),\,\, ({\boldsymbol{x}}, t) \in \Omega \times(0,T], 
\end{array} \\
\begin{array}{ll}
u({\boldsymbol{x}}, t)=0, & ({\boldsymbol{x}}, t) \in\left(\mathbb{R}^3 \backslash \Omega\right) \times[0,T], \\
u({\boldsymbol{x}}, 0)=u_0({\boldsymbol{x}}), & {\boldsymbol{x}} \in \Omega,
\end{array}\end{cases}
\end{equation}
in which $0<\alpha,\beta,\gamma<1$, ${\boldsymbol{x}}=(x,y,z) \in \Omega:=(a_1, b_1) \times(a_2, b_2)\times(a_3,b_3)$. Here $u({\boldsymbol{x}},t)=u(x,y,z,t)$ represents concentration, mass or other physical quantities of interest, and $f({\boldsymbol{x}}, t)$ is a source term. Moreover, 
\begin{equation}
\mathcal{D}_{x}^{\alpha}u({\boldsymbol{x}}, t)=k_{1,+} \frac{\partial^{1-\alpha}}{\partial x^{1-\alpha}}u({\boldsymbol{x}}, t)-k_{1,-} \frac{\partial^{1-\alpha}}{\partial (-x)^{1-\alpha}}u({\boldsymbol{x}}, t),
\end{equation}
where $k_{1,\pm}$ are two positive constant diffusivity coefficients; $\frac{\partial^{1-\alpha} u(x, y,z,t)}{\partial x^{1-\alpha}}$ and $\frac{\partial^{1-\alpha} u(x, y, z, t)}{\partial (x)^{1-\alpha}}$ denote the left and right R-L fractional derivatives, respectively, which are defined by
\begin{equation}
\frac{\partial^{1-\alpha} u(x, y, z, t)}{\partial x^{1-\alpha}}:=\frac{1}{\Gamma(\alpha)}\frac{\partial}{\partial x} \int_{a_1}^x(x-s)^{\alpha-1} u(s,y,z,t) ds,
\end{equation}
and
\begin{equation}
\frac{\partial^{1-\alpha} u(x, y, z, t)}{\partial (-x)^{1-\alpha}}:=-\frac{1}{\Gamma(\alpha)}\frac{\partial}{\partial x} \int_x^{b_1}(s-x)^{\alpha-1} u(s,y,z,t) ds,
\end{equation}
with $\Gamma(\cdot)$ being the gamma function. $\mathcal{D}_{y}^{\beta}u({\boldsymbol{x}}, t)$ and $\mathcal{D}_{z}^{\gamma}u({\boldsymbol{x}}, t)$ can be analogously defined. 

Similar to their integer-order counterparts, FDEs rarely yield analytical solutions, making numerical approaches essential. In recent years, significant progresses have been made in this area, including finite difference (FD) methods \cite{CD2,LL1,lin2021stability,she2022class,SL1,TZD1,ZD1}, finite element methods \cite{bu2014galerkin,JLPZ1,zhao2015finite}, and finite volume (FV) methods \cite{donatelli2018spectral,fang2022fast,feng2015stability,fu2019finite,fu2019stability,hejazi2014stability,kong2024fast,LZTBA1,liu2021analysis,pan2017fast,pan2016fast,wang2013superfast,simmons2017finite,zheng2019efficient}. 

However, unlike classical diffusion equations, the nonlocal nature of fractional derivatives typically yields full and dense discrete linear systems that require $\mathcal{O}(N^2)$ memory storage and $\mathcal{O}(N^3)$ computational cost for the matrix size $N$. Consequently, the computational cost for large-scale simulations of FDEs become prohibitively expensive. Notably, the discrete linear systems derived from the FD and FV methods exhibit Toeplitz-like structures. 
By leveraging these structures, the matrix-vector multiplications in the Krylov subspace method can be efficiently computed using fast Fourier transforms (FFTs). However, numerical experiments demonstrate that the ill-conditioning of the coefficient matrix increasingly degrades the performance of the iterative solvers as the matrix size $N$ grows. In such cases, preconditioning techniques are employed to accelerate the convergence of Krylov subspace methods. Classical preconditioners, including circulant preconditioners \cite{lei2016multilevel,lei2013circulant}, banded preconditioners \cite{JLZ1,LYJ1,SLYL1}, approximate inverse preconditioners \cite{PKNS1,zeng2022tau}, structure preserving preconditioners \cite{DMS1}, and $\tau$-preconditioners \cite{huang2021spectral,huang2023tau,lin2023tau,qu2025novel}, have been proposed to accelerate the solution of the discrete linear systems arising from FD discretization.

Prior research has established that FD method can be applied directly to spatial FDEs in non-conservative form, whereas FV method is typically employed for conservative forms such as Equation (\ref{RLFDEs}). Furthermore, FV method exhibits higher derivation complexity than FD method due to its dependence on grid-node basis functions, resulting in more complicated coefficient matrix structures. To address these issues, developing efficient fast algorithms for solving conservative SFDEs becomes imperative. 

In early work, Wang and Du \cite{wang2013superfast} have studied a superfast-preconditioned conjugate gradient squared method for one-dimensional (1D) steady-state Riesz space FDEs. Their method reduces the computational complexity from $\mathcal{O}(N^2)$ to $\mathcal{O} (N\log N)$ per iteration and reduces the memory requirement from $\mathcal{O}(N^2)$ to $\mathcal{O}(N)$. Subsequently, Pan et al. \cite{pan2017fast,pan2016fast} extended this problem to R-L space FDEs in two dimensions, including both time-independent and time-dependent variable-coefficient cases. They respectively proposed a scaled-circulant preconditioner and an approximate inverse-circulant preconditioner for solving the linear systems arising from the FV discretization. Theoretical analyses demonstrate that for the time-independent case,
the difference between the scaled-circulant preconditioner and the coefficient matrix  equals the sum of a small-norm matrix and a low-rank matrix, while for the time-dependent problem, the spectra of the corresponding preconditioned matrices are clustered around $1$. Then rapid convergence of the preconditioned Krylov subspace methods can be expected with these properties. In \cite{donatelli2018spectral}, Donatelli et al. developed a finite volume discretization for 1D and two-dimensional (2D) conservative steady-state FDEs with variable coefficients. They revealed the intrinsic nature of the resulting sequence of coefficient matrices through the framework of generalized locally Toeplitz sequences. This analysis provides essential spectral information for designing efficient preconditioners and multigrid
methods to solve the resulting linear systems. Subsequently, Fu's group \cite{fu2019finite,fu2019stability} combined Crank-Nicolson (CN) and FV methods for solving 1D and 2D conservative SFDEs. They established the stability and convergence of the CN-FV scheme through rigorous norm analysis. To improve computational efficiency, circulant preconditioners were designed to solve the resulting linear systems, although the spectral analysis of these preconditioned matrices was not addressed. Additionally, the CN-FV scheme has been effectively extended to 3D nonlinear distributed-order conservative SFDEs, see \cite{zheng2019efficient}. Furthermore, to handle high-dimensional conservative FDEs, the alternating direction implicit method can be employed. For more details, readers are referred to \cite{chou2021finite,liu2021analysis} and the references therein. 

The aforementioned works demonstrate the strong potential of the fast algorithms for conservative SFDEs (\ref{RLFDEs}). To the best of our knowledge, there is still a lack of sine transform-based preconditioning techniques
for conservative SFDEs. In this paper, based on the CN-FV scheme developed by Fu's group \cite{fu2019finite,zheng2019efficient}, we propose a novel sine transform-based preconditioning strategy combined with the CG and GMRES methods to solve the symmetric and non-symmetric discrete linear systems arising from conservative SFDEs. The aim of this paper is twofold: 
\begin{enumerate}
    \item For the symmetric case, we propose a sine transform-based preconditioner, which can be implemented efficiently through the fast discrete sine transform (DST), for solving the symmetric CN-FV scheme (\ref{linear_sys}). As will be shown in Theorem \ref{CG_bounded_sprctrum}, the CG method is optimal in the sense that the spectra of the preconditioned matrices are uniformly bounded in the open interval $(1/2, 3/2)$.
    \item For the non-symmetric case, the preconditioned GMRES (PGMRES) method with the proposed sine transform-based preconditioner is employed to solve the auxiliary two-sided preconditioned systems (\ref{two_side_system}). We show the GMRES method for the preconditioned system (\ref{two_side_system}) has a convergence rate independent of the matrix size with the guidance from Lemma \ref{lemma:gmres}. In more detail, two points are thoroughly discussed: (i) estimating the upper and lower bounds of the eigenvalues for the Hermitian part of the coefficient matrix in the two-sided preconditioned systems (\ref{two_side_system}), see Lemma \ref{prop:eigen_S}; (ii) evaluating the upper bound of spectral radius for the skew-Hermitian part of the coefficient matrix by deriving the relationship between the real and imaginary parts of the generating function for the FV discretization matrix with a new technique, see Lemmas \ref{non_symmetric_ratio} and \ref{lemma:eigen_Skew}. Finally, the relationship between the residuals of GMRES method when applied to one-sided  and two-sided preconditioned systems are provided to support the optimal convergence of GMRES method for the one-sided preconditioned systems (\ref{one_side_system}).
    
\end{enumerate}

The paper is organized as follows. In Section \ref{Sec-2}, we introduce the CN-FV method to discretize the conservative SFDEs (\ref{RLFDEs}), and present some important properties of the coefficients obtained from FV approximations. In Section \ref{Sec-3}, we propose the sine transform-based preconditioning techniques for both symmetric and non-symmetric cases and provide rigorous theoretical  analyses that guarantee the mesh size independent convergence rates of the preconditioned Krylov subspace methods. In Section \ref{sec:num}, numerical experiments containing both symmetric and non-symmetric cases in 2D and 3D problems are presented to demonstrate the efficiency and robustness the proposed preconditioning strategies. Finally, we draw some conclusions in Section \ref{sec:con}.


\section{Crank-Nicolson and finite volume approximations}\label{Sec-2}
In this section, we first review the CN-FV approximations to conservative SFDEs (\ref{RLFDEs}). Let $\Delta t=\frac{T}{M}$, with $M$ being a positive integer, be the temporal step size under the uniform partition of $[0, T]$, and define $t_m:=m\Delta t$ for $m=0,1,\ldots,M$. Discretizing Equation \eqref{RLFDEs} in time at $t=t_{m-\frac{1}{2}}$ by the CN method arrives at 
\begin{eqnarray}\label{CN}
&&u({\boldsymbol{x}},t_{m})-\frac{\Delta t}{2}\left(\frac{\partial}{\partial x} \mathcal{D}_{x}^\alpha u({\boldsymbol{x}}, t_m)+\frac{\partial}{\partial y} \mathcal{D}_{y}^\beta u({\boldsymbol{x}}, t_m)+\frac{\partial}{\partial z} \mathcal{D}_{z}^\gamma u({\boldsymbol{x}}, t_m)\right)\\
&=&u({\boldsymbol{x}},t_{m-1})+\frac{\Delta t}{2}\left(\frac{\partial}{\partial x} \mathcal{D}_{x}^\alpha u({\boldsymbol{x}}, t_{m-1})+\frac{\partial}{\partial y} \mathcal{D}_{y}^\beta u({\boldsymbol{x}}, t_{m-1})+\frac{\partial}{\partial z} \mathcal{D}_{z}^\gamma u({\boldsymbol{x}}, t_{m-1})\right)\nonumber\\
&&+\Delta tf({\boldsymbol{x}}, t_{m-\frac{1}{2}})+\mathcal{O}(\Delta t^3),~~~~ \quad 1 \leq m \leq M. \nonumber
\end{eqnarray}

Furthermore, let $n_i\,(i=1,2,3)$ be three positive integers, and denote 
\begin{eqnarray*}
 h_1:=\frac{b_1-a_1}{n_1+1},\, x_p:=a_1+ph_1,\, 0\leq p\leq n_1+1,\\
 h_2:=\frac{b_2-a_2}{n_2+1},\, y_q:=a_2+qh_2,\,0\leq q\leq n_2+1,\\  
 h_3:=\frac{b_3-a_3}{n_3+1},\, z_r:=a_3+rh_3,\,0\leq r\leq n_3+1,
\end{eqnarray*}
where $h_i\,(i=1,2,3)$ are the spatial steps, and $x_p,y_q,z_r$ are the sets of grid points. Denote $S_h(\Omega)$ be the space of continuous, piecewise linear functions defined on the spatial partition, satisfying the zeros boundary condition on $\partial \Omega$. Then, the FV solution $u_h\left(\boldsymbol{x},t_m\right)$ of problem \eqref{RLFDEs} can be expressed in terms of the nodal basis functions
\begin{equation}\label{uh}
u_h\left(\boldsymbol{x},t_m\right):=\sum_{p=1}^{n_1} \sum_{q=1}^{n_2}\sum_{r=1}^{n_3} u_{p, q, r}^m \phi_{p}^x(x) \phi_{q}^y(y) \phi_{r}^z(z),\, m=1,2,\ldots,M,
\end{equation}
where $u_{p, q, r}^m$ be the FV approximations to $u(x_{p},y_{q},z_{r},t_m)$ for $p=1,\ldots,n_1$, $q=1,\ldots,n_2$, $r=1,\ldots,n_3$, and $\phi_{p}^x(x)$, $\phi_{q}^y(y)$ and $\phi_{r}^z(z)$ are the nodal basis functions in $x$-, $y$- and $z$-directions, respectively, with
\begin{equation*}
\begin{aligned}
&\begin{aligned}
& \phi_0^x(x)= \begin{cases}\frac{x_1-x}{h_1}, & x \in\left[a_1, x_1\right], \\
0, & \text { elsewhere, }\end{cases} \\
& \phi_i^x(x)=\left\{\begin{array}{ll}
\frac{x-x_{i-1}}{h_1}, & x \in\left[x_{i-1}, x_i\right], \\
\frac{x_{i+1}-x}{h_1}, & x \in\left[x_i, x_{i+1}\right], \\
0, & \text { elsewhere, }
\end{array} \quad 1 \leq i \leq n_1,\right. \\
& \phi_{n_1+1}^x(x)= \begin{cases}\frac{x-x_{n_1}}{h_1}, & x \in\left[x_{n_1}, b_1\right], \\
0, & \text { elsewhere}.\end{cases}
\end{aligned}\\
\end{aligned}
\end{equation*}
$\phi_{q}^y(y)$ and $\phi_{r}^z(z)$ can be defined similarly.

Now, by replacing the true solution $u\left(\boldsymbol{x}, t_m\right)$ in Equation (\ref{RLFDEs}) with the FV solution $u_h\left(\boldsymbol{x}, t_m\right)$, integrating both sides of the governing Equation (\ref{CN}) over $\Omega_{i, j, k}:=\left[x_{i-1/2}, x_{i+1/2}\right] \times\left[y_{j-1/2}, y_{j+1/2}\right]\times\left[z_{k-1/2}, z_{k+1/2}\right]$ for $1 \leq i \leq n_1, 1 \leq j \leq n_2, 1\leq k \leq n_3$ with $u\left(x, y, z, t_m\right)$ and omitting the temporal truncation error term, we obtain the following CN-FV scheme
for model (\ref{RLFDEs}):
\begin{eqnarray}\label{fvm_sol}
& &\sum_{p=1}^{n_1} \sum_{q=1}^{n_2} \sum_{r=1}^{n_3} u_{p, q, r}^m \int_{x_{i-1/2}}^{x_{i+1/2}} \int_{y_{j-1/2}}^{y_{j+1/2}} \int_{z_{k-1/2}}^{z_{k+1/2}}\phi_p^x(x) \phi_q^y(y) \phi_r^z(z) dxdydz \nonumber\\
& &\quad-\left.\frac{\Delta t}{2} \sum_{p=1}^{n_1} \sum_{q=1}^{n_2} \sum_{r=1}^{n_3} u_{p, q, r}^m\left(\mathcal{D}_{x}^\alpha \phi_p^x(x)\right)\right|_{x_{i-1/2}} ^{x_{i+1/2}} \int_{y_{j-1/2}}^{y_{j+1/2}} \phi_q^y(y) d y \int_{z_{k-1/2}}^{z_{k+1/2}} \phi_r^z(z) d z\nonumber\\
& & \quad-\left.\frac{\Delta t}{2} \sum_{p=1}^{n_1} \sum_{q=1}^{n_2} \sum_{r=1}^{n_3} u_{p, q, r}^m\left(\mathcal{D}_{y}^\beta \phi_q^y(y)\right)\right|_{y_{j-1/2}} ^{y_{j+1/2}} \int_{x_{i-1 / 2}}^{x_{i+1 / 2}} \phi_p^x(x) d x 
\int_{z_{k-1/2}}^{z_{k+1/2}} \phi_r^z(z) d z\nonumber\\
& &\quad-\left.\frac{\Delta t}{2} \sum_{p=1}^{n_1} \sum_{q=1}^{n_2} \sum_{r=1}^{n_3} u_{p, q, r}^m\left(\mathcal{D}_{z}^\gamma \phi_r^z(z)\right)\right|_{z_{k-1/2}} ^{z_{k+1/2}} \int_{x_{i-1 / 2}}^{x_{i+1 / 2}} \phi_p^x(x) d x 
\int_{y_{j-1/2}}^{y_{j+1/2}} \phi_q^y(y) d y\nonumber\\
&=& \sum_{p=1}^{n_1} \sum_{q=1}^{n_2} \sum_{r=1}^{n_3} u_{p, q, r}^{m-1} \int_{x_{i-1/2}}^{x_{i+1/2}} \int_{y_{j-1/2}}^{y_{j+1/2}} \int_{z_{k-1/2}}^{z_{k+1/2}}\phi_p^x(x) \phi_q^y(y) \phi_r^z(z) dxdydz \\
& & \quad+\left.\frac{\Delta t}{2} \sum_{p=1}^{n_1} \sum_{q=1}^{n_2} \sum_{r=1}^{n_3} u_{p, q, r}^{m-1}\left(\mathcal{D}_{x}^\alpha \phi_p^x(x)\right)\right|_{x_{i-1/2}} ^{x_{i+1/2}} \int_{y_{j-1/2}}^{y_{j+1/2}} \phi_q^y(y) d y \int_{z_{k-1/2}}^{z_{k+1/2}} \phi_r^z(z) d z\nonumber\\
& & \quad+\left.\frac{\Delta t}{2} \sum_{p=1}^{n_1} \sum_{q=1}^{n_2} \sum_{r=1}^{n_3} u_{p, q, r}^{m-1}\left(\mathcal{D}_{y}^\beta \phi_q^y(y)\right)\right|_{y_{j-1/2}} ^{y_{j+1/2}} \int_{x_{i-1/2}}^{x_{i+1/2}} \phi_p^x(x) d x 
\int_{z_{k-1/2}}^{z_{k+1/2}} \phi_r^z(z) d z\nonumber\\
& & \quad+\left.\frac{\Delta t}{2} \sum_{p=1}^{n_1} \sum_{q=1}^{n_2} \sum_{r=1}^{n_3} u_{p, q, r}^{m-1}\left(\mathcal{D}_{z}^\gamma \phi_r^z(z)\right)\right|_{z_{k-1/2}} ^{z_{k+1/2}} \int_{x_{i-1/2}}^{x_{i+1/2}} \phi_p^x(x) d x 
\int_{y_{j-1/2}}^{y_{j+1/2}} \phi_q^y(y) d y\nonumber\\
& & \quad+\Delta t \int_{x_{i+1/2}}^{x_{{i+1 / 2}}} \int_{y_{j-1/2}}^{y_{{j+1/2}}}\int_{z_{k-1/2}}^{z_{{k+1/2}}} f\left(x, y, z, t_{m-1 / 2}\right) dxdydz.\nonumber
\end{eqnarray}

By some straightforward computations, for $p=1,2,\ldots,n_1$, we have 
\begin{equation}
\int_{x_{i-1 / 2}}^{x_{i+1 / 2}} \phi_p^x(x) d x=\frac{h_1}{8} \begin{cases}1 & |p-i|=1, \\ 6, & p=i, \\ 0, & \text { otherwise. }\end{cases}
\end{equation}
and 
\begin{equation}
\begin{aligned}
\left.\frac{\partial^{1-\alpha} \phi_p^x(x)}{\partial x^{1-\alpha}}\right|_{x=x_{i-1/2}} & =\frac{1}{h_1^{1-\alpha} \Gamma(\alpha+1)} \begin{cases}0, & p>i, \\
s_{i-p}^{(\alpha)}, & p \leq i,\end{cases} \\
\left.\frac{\partial^{1-\alpha} \phi_p^x(x)}{\partial x^{1-\alpha}}\right|_{x=x_{i+1 / 2}} & =\frac{1}{h_1^{1-\alpha} \Gamma(\alpha+1)} \begin{cases}0, & p>i+1, \\
s_{i-p+1}^{(\alpha)}, & p \leq i+1,\end{cases} \\
\left.\frac{\partial^{1-\alpha} \phi_p^x(x)}{\partial (-x)^{1-\alpha}}\right|_{x=x_{i-1 / 2}} & =\frac{1}{h_1^{1-\alpha} \Gamma(\alpha+1)} \begin{cases}s_{p-i+1}^{(\alpha)}, & p \geq i-1, \\
0, & p<i-1,\end{cases} \\
\left.\frac{\partial^{1-\alpha} \phi_p^x(x)}{\partial (-x)^{1-\alpha}}\right|_{x=x_{i+1 / 2}} & =\frac{1}{h_1^{1-\alpha} \Gamma(\alpha+1)} \begin{cases}s_{p-i}^{(\alpha)}, & p \geq i, \\
0, & p<i,\end{cases}
\end{aligned}
\end{equation}
where
\begin{equation}
s_i^{(\alpha)}= \begin{cases}\left(\frac{1}{2}\right)^{\alpha}, & i=0, \\ \left(\frac{3}{2}\right)^{\alpha}-2\left(\frac{1}{2}\right)^{\alpha}, & i=1, \\ \left(i+\frac{1}{2}\right)^{\alpha}-2\left(i-\frac{1}{2}\right)^{\alpha}+\left(i-\frac{3}{2}\right)^{\alpha}, & 2 \leq i \leq n_1.
\end{cases}
\end{equation}
It is remarked that the above results are also valid for $\{\phi_q^y(y)\}_{q=1}^{n_2}$ and $\{\phi_r^z(z)\}_{r=1}^{n_3}$ with $h_1$ replaced by $h_2$ and $h_3$ and $\alpha$ replaced by $\beta$ and $\gamma$, respectively.

Define $N:=n_1n_2n_3$ dimensional vectors $\mathbf{u}^m$ and $\mathbf{F}^{m-1/2}$ by labeling the index in the $x$-direction first, $y$-direction second, and $z$-direction last, that is,
\begin{equation}
\mathbf{u}^m=\left[u_{1,1,1}^m, \ldots, u_{n_1, 1,1}^m, \ldots, u_{1, n_2, 1}^m, \ldots, u_{n_1, n_2, 1}^m, \ldots, u_{1,1, n_3}^m, \ldots, u_{n_1, n_2, n_3}^m\right]^{\top},
\end{equation}
and 
\begin{equation}
\mathbf{F}^{m-1/2}=\left[f_{1,1,1}^{m-1/2}, \ldots, f_{n_1, 1,1}^{m-1/2}, \ldots, f_{1, n_2, 1}^{m-1/2}, \ldots, f_{n_1, n_2, 1}^{m-1/2}, \ldots, f_{1,1, n_3}^{m-1/2}, \ldots, f_{n_1, n_2, n_3}^{m-1/2}\right]^{\top},
\end{equation}
with 
\begin{equation}
{f}_{i, j, k}^{m-1/2}=\frac{1}{h_1 h_2  h_3} \int_{x_{i-1/2}}^{x_{i+1/2}} \int_{y_{j-1/2}}^{y_{j+1/2}} \int_{z_{k-1 / 2}}^{z_{k+1/2}}f\left(x, y, z, t_{m-1/2}\right)dzdydx
\end{equation}

For simplicity, we denote $\delta_i= \begin{cases}\alpha, & i=1 \\ \beta, & i=2\\ \gamma, &  i =3 \end{cases},$ then
the CN-FV scheme (\ref{fvm_sol}) can be reformulated in the matrix form given in \cite{fu2019finite,zheng2019efficient}
\begin{eqnarray}\label{original_matrix_form}
\begin{footnotesize}
\begin{aligned}
& \left(\mathbf{A}_{n_3} \otimes \mathbf{A}_{n_2} \otimes \mathbf{A}_{n_1}+\eta_\alpha \mathbf{A}_{n_3} \otimes \mathbf{A}_{n_2} \otimes \mathbf{B}_{n_1}^{\alpha}+\eta_\beta \mathbf{A}_{n_3} \otimes \mathbf{B}_{n_2}^{\beta} \otimes \mathbf{A}_{n_1}+\eta_\gamma \mathbf{B}_{n_3}^{\gamma} \otimes \mathbf{A}_{n_2} \otimes \mathbf{A}_{n_1}\right) \mathbf{u}^m \\
&=\left(\mathbf{A}_{n_3} \otimes \mathbf{A}_{n_2} \otimes \mathbf{A}_{n_1}-\eta_\alpha \mathbf{A}_{n_3} \otimes \mathbf{A}_{n_2} \otimes \mathbf{B}_{n_1}^{\alpha}-\eta_\beta \mathbf{A}_{n_3} \otimes \mathbf{B}_{n_2}^{\beta} \otimes \mathbf{A}_{n_1}-\eta_\gamma \mathbf{B}_{n_3}^{\gamma} \otimes \mathbf{A}_{n_2} \otimes \mathbf{A}_{n_1}\right)  \mathbf{u}^{m-1}+\tau \mathbf{F}^{m-1/2},
\end{aligned}
\end{footnotesize}
\end{eqnarray}
with $m=1,2,\ldots M$. For $i=1,2,3$, $\eta_{\delta_i}=\frac{\Delta t}{2\Gamma(\delta_i+1)h_i^{2-\delta_i}}$, $\mathbf{A}_{n_i}$ is a $n_i$-by-$n_i$ symmetric tridiagonal Toeplitz matrix defined as 
\begin{equation}\label{tridiag_matrix}
\mathbf{A}_{n_i}=\frac{1}{8}\left[\begin{array}{ccccc}
6 & 1 & 0 & \cdots & 0 \\
1 & 6 & 1 & \ddots & \vdots \\
0 & 1 & 6 & \ddots & 0 \\
\vdots & \ddots & \ddots & \ddots & 1 \\
0 & \cdots & 0 & 1 & 6
\end{array}\right] \in \mathbb{R}^{n_i\times n_i},
\end{equation}
and $\mathbf{B}_{n_i}^{\delta_i}$ of order $n_i$ is a stiffness matrix, which is expressed as 
\begin{equation}\label{stiff_matrix}
\mathbf{B}_{n_i}^{\delta_i}=k_{i,+}\mathbf{T}^{\delta_i}_{n_i}+k_{i,-}\left(\mathbf{T}^{\delta_i}_{n_i}\right)^{\top}
\end{equation}
where $\mathbf{T}^{\delta_i}_{n_i}$ is a Toeplitz matrix such that 
\begin{equation}\label{T_ni}
\mathbf{T}_{n_i}^{\delta_i}:=\left[\begin{array}{llllll}
q_1^{(\delta_i)} & q_0^{(\delta_i)} & 0 & \cdots & 0\\
q_2^{(\delta_i)} & q_1^{(\delta_i)} & q_0^{(\delta_i)} & \ddots & \vdots\\
\vdots & \ddots & \ddots & \ddots & 0\\
\vdots & \ddots & \ddots & \ddots & q_0^{(\delta_i)}\\
q_{n_i}^{(\delta_i)} & q_{n_i-1}^{(\delta_i)} & \cdots &q_2^{(\delta_i)} & q_1^{(\delta_i)}
\end{array}\right] \in \mathbb{R}^{n_i\times n_i},
\end{equation}
with
\begin{equation}\label{coef_expression}
q_k^{(\delta_i)}= 
\begin{cases}-s_0^{(\delta_i)}, & k=0, \\ s_{k-1}^{(\delta_i)}-s_k^{(\delta_i)}, & 1 \leq k \leq n_i,
\end{cases}
\end{equation}
in which 
\begin{equation}
s_k^{(\delta_i)}= \begin{cases}\left(\frac{1}{2}\right)^{\delta_i}, & k=0, \\ \left(\frac{3}{2}\right)^{\delta_i}-2\left(\frac{1}{2}\right)^{\delta_i}, & k=1, \\ \left(k+\frac{1}{2}\right)^{\delta_i}-2\left(k-\frac{1}{2}\right)^{\delta_i}+\left(k-\frac{3}{2}\right)^{\delta_i}, & 2 \leq k \leq n_i.
\end{cases}
\end{equation}
\begin{remark}
{\rm
As discussed in \cite{fu2019finite,zheng2019efficient}, the scheme in (\ref{original_matrix_form}) can be shown to be unconditionally stable and convergent, achieving second-order accuracy in time and $\min\{1+\alpha,1+\beta,1+\gamma\}$ accuracy in space with respect to a weighted discrete norm. In particular, when $k_{i,+}=k_{i,-}\,(i=1,2,3)$, (\ref{RLFDEs}) reduces to a symmetry problem, then the spatial accuracy of the scheme (\ref{original_matrix_form}) can be improved to second-order, as detailed in \cite{fu2019stability}. 
}
\end{remark}

For any real symmetric matrices $\mathbf{C}_1, \mathbf{C}_2 \in \mathbb{R}^{k \times k}$, denote $\mathbf{C}_2 \succ$ (or $\succeq$) $\mathbf{C}_1$ if $\mathbf{C_2-C_1}$ is positive definite (or semi-definite). Especially, we denote $\mathbf{C}_2 \succ$ (or $\succeq$) $\mathbf{O}$, if $\mathbf{C}_2$ itself is positive definite
(or semi-definite). Also, $\mathbf{C}_1 \prec$ (or $\preceq$) $\mathbf{C}_2$ and $\mathbf{O} \prec$ (or $\preceq$) $\mathbf{C}_2$  have the same meanings as those
of $\mathbf{C}_2 \succ$ (or $\succeq$) $\mathbf{C}_1$ and $\mathbf{C}_2 \succ$ (or $\succeq$) $\mathbf{O}$, respectively.

In addition, we denote the symmetric part and the skew-symmetric part of a real square matrix $\mathbf{Z}$ as
$$
\mathcal{H}(\mathbf{Z}):=\frac{\mathbf{Z}+\mathbf{Z}^{\top}}{2} \quad \textrm{and} \quad \mathcal{S}(\mathbf{Z}):=\frac{\mathbf{Z}-\mathbf{Z}^{\top}}{2}.
$$
Furthermore, denote
\begin{equation*}
\mathbf{A}_N:=\underset{i=3}{\overset{1}{\otimes}}
\mathbf{A}_{n_i},~~\mathbf{A}_{n_1^{+}}:=\mathbf{A}_{n_3} \otimes \mathbf{A}_{n_2},~~\mathbf{A}_{n_3^{-}}:=\mathbf{A}_{n_2} \otimes \mathbf{A}_{n_1}.
\end{equation*}

Then the matrix form of (\ref{original_matrix_form}) can be expressed as the simplified form
\begin{equation}\label{linear_sys}
\begin{small}
\begin{aligned}
\mathbf{A}\mathbf{u}^m &:= \left(\mathbf{A}_N +\eta_\alpha \mathbf{A}_{n_1^{+}} \otimes \mathbf{B}_{n_1}^{\alpha}+\eta_\beta \mathbf{A}_{n_3} \otimes \mathbf{B}_{n_2}^{\beta} \otimes \mathbf{A}_{n_1}+\eta_\gamma \mathbf{B}_{n_3}^{\gamma} \otimes \mathbf{A}_{n_3^{-}}\right) \mathbf{u}^m \\
& =\left(\mathbf{A}_N -\eta_\alpha \mathbf{A}_{n_1^{+}} \otimes \mathbf{B}_{n_1}^{\alpha}-\eta_\beta \mathbf{A}_{n_3} \otimes \mathbf{B}_{n_2}^{\beta} \otimes \mathbf{A}_{n_1}-\eta_\gamma \mathbf{B}_{n_3}^{\gamma} \otimes \mathbf{A}_{n_3^{-}}\right)   \mathbf{u}^{m-1}+\tau \mathbf{F}^{m-1/2}\\
&=:\mathbf{b}^{m},
\end{aligned}
\end{small}
\end{equation}

\begin{remark}
{\rm
Note that when $k_{i,+}=k_{i,-}\,(i=1,2,3)$, scheme (\ref{linear_sys}) reduces to a symmetric positive definite (SPD) problem. In this case, the CG method is more preferred to solve this kind of large-scale and dense linear systems. Otherwise, the non-symmetric systems will be solved by the GMRES method.
}
\end{remark}

The properties of the elements $q_i^{(\delta_i)}$ in \eqref{T_ni} are presented in the following two lemmas, which are essential for analyzing the convergence behaviors of the CG and GMRES methods in the subsequent section.

\begin{lemma}\label{coef_prop1}\cite{fu2019stability,qu2021fast}
For all $0<\delta_i<1$, and let $\{q_{k}^{(\delta_i)}\}_{k=0}^{n_i}$ be defined as (\ref{coef_expression}). Then we have
\begin{description}
  \item[(i)]  $q_1^{(\delta_i)}>0$, $q_0^{(\delta_i)}+q_2^{(\delta_i)}<0$, and $q_k^{(\delta_i)}<0$ for $k=3,\cdots$;
  \item[(ii)]  $\sum\limits_{k=0}^{\infty}q_{k}^{(\delta_i)}=0$;
  \item[(iii)]  $q_{1}^{(\delta_i)}+\sum\limits_{k=0,k\ne1}^{s}q_{k}^{(\delta_i)}>0$ for $s\geq 2$.
\end{description}
\end{lemma}

\begin{lemma}\label{coef_prop2}
For all $0<\delta_i<1$, and let $q_{k}^{(\delta_i)}$ be defined as (\ref{coef_expression}). Then we have
\begin{equation}
q_0^{(\delta_i)}+q_2^{(\delta_i)}<q_3^{(\delta_i)}<\cdots<q_{n_i}^{(\delta_i)}<\cdots.
\end{equation}
\end{lemma}
\begin{proof}
Firstly, by direct calculation, we obtain
\begin{eqnarray}
h({\delta_i}):=q_3^{({\delta_i})}-(q_0^{({\delta_i})}+q_2^{({\delta_i})})&=&2s_2^{({\delta_i})}-s_3^{({\delta_i})}-s_1^{({\delta_i})}+s_0^{({\delta_i})}\\ \nonumber
&=&-\left(\frac{7}{2}\right)^{\delta_i}+4\left(\frac{5}{2}\right)^{\delta_i}-6\left(\frac{3}{2}\right)^{\delta_i}+5\left(\frac{1}{2}\right)^{\delta_i}.   \nonumber
\end{eqnarray}
As $h({\delta_i})$ strictly decreases with $\delta_i$ and $h(1)=0$, it follows that $h(\delta_i)>h(1)$, that is, $q_3^{({\delta_i})}>q_0^{({\delta_i})}+q_2^{({\delta_i})}$.

In addition, define $g(x)=\left(x+\frac{1}{2}\right)^{\delta_i}-2\left(x-\frac{1}{2}\right)^{\delta_i}+\left(x-\frac{3}{2}\right)^{\delta_i},~x\geq2$.
Since ${\delta_i}-1<0$ and $\phi(y)=y^{{\delta_i}-2}$ is convex for $y>0$, the second derivative of $g(x)$ satisfies
$g^{''}(x)={\delta_i}({\delta_i}-1)\left(\left(x+\frac{1}{2}\right)^{{\delta_i}-2}-2\left(x-\frac{1}{2}\right)^{{\delta_i}-2}+\left(x-\frac{3}{2}\right)^{{\delta_i}-2}\right)<0$, implying $g(x)$ is concave. Thus, for $k\geq3$,
\begin{eqnarray}
q_{k+1}^{({\delta_i})}-q_k^{({\delta_i})}&=&2s_{k}^{({\delta_i})}-s_{k+1}^{(\delta)}-s_{k-1}^{({\delta_i})}\\  \nonumber
&=&2g(k)-g(k+1)-g(k-1)\\  \nonumber
&>&0.   \nonumber
\end{eqnarray}
The proof is complete.
\end{proof}

\section{Preconditioners and convergence behaviors}\label{Sec-3}

In general, the iterative method converges slowly since the coefficient matrix will be ill-conditioned as the spatial grid mesh increases. Hence, some specific techniques, such as preconditioning, are necessary to improve the proposed fast iterative method by reducing the iteration number. Recently, sine transform-based preconditioners have been widely used to precondition both symmetric and non-symmetric Toeplitz systems derived from FD methods, as illustrated in \cite{hon2024symbol,huang2021spectral,huang2024optimal,huang2023tau,li2025multilevel,lin2023tau,qin2023sine,qu2025novel,zhang2022fast} and the references therein. Inspired by this novel
preconditioning idea, we aim at constructing the following preconditioner for solving the CN-FV scheme (\ref{linear_sys}): 
\begin{equation}\label{pre_definition}
\begin{small}
\begin{aligned}
\mathbf{P}&=\mathbf{A}_N +\eta_\alpha \mathbf{A}_{n_1^{+}} \otimes \tau(\mathcal{H}(\mathbf{B}_{n_1}^{\alpha}))+\eta_\beta \mathbf{A}_{n_3} \otimes \tau(\mathcal{H}(\mathbf{B}_{n_2}^{\beta})) \otimes \mathbf{A}_{n_1}+\eta_\gamma \tau(\mathcal{H}(\mathbf{B}_{n_3}^{\gamma})) \otimes \mathbf{A}_{n_3^{-}}
 \\  
&=\mathbf{A}_N +\eta_\alpha (k_{1,+}+k_{1,-}) \mathbf{A}_{n_1^{+}} \otimes \tau(\mathcal{H}(\mathbf{T}_{n_1}^{\alpha}))+\eta_\beta (k_{2,+}+k_{2,-}) \mathbf{A}_{n_3} \otimes \tau(\mathcal{H}(\mathbf{T}_{n_2}^{\beta})) \otimes \mathbf{A}_{n_1}\\ 
&~~~+\eta_\gamma (k_{3,+}+k_{3,-}) \tau(\mathcal{H}(\mathbf{T}_{n_3}^{\gamma})) \otimes \mathbf{A}_{n_3^{-}}   
\end{aligned}
\end{small}
\end{equation}
where $\tau(\mathcal{H}(\mathbf{T}_{n_i}^{\delta_i}))$ is the $\tau$ preconditioner \cite{bini1990new} for $\mathcal{H}(\mathbf{T}_{n_i}^{\delta_i})$. Note that $\tau(\mathcal{H}(\mathbf{T}_{n_i}^{\delta_i}))$ can be fast diagonalized by the sine transform matrix $\mathbf{S}_{n_i}$ with $[\mathbf{S}_{n_i}]_{j,k}=\sqrt{\frac{2}{n_i+1}}\sin\left(\frac{jk\pi}{n_i+1}\right)$ for $1\leq j,k\leq n_i$; and the eigenvalues of $\mathcal{H}(\mathbf{T}_{n_i}^{\delta_i})$ can be determined by the first column elements of $\mathcal{H}(\mathbf{T}_{n_i}^{\delta_i})$ \cite{bini1990new,huang2023tau}. Since the symmetric tridiagonal matrix $\mathbf{A}_{n_i}$ is also a $\tau$ matrix, the preconditioner $\mathbf{P}$ can be diagonalized by the 3-level sine transform matrix $\mathbf{S}_{N}$, i.e.,
\begin{equation}\label{diag_tau_pre}
\begin{aligned}
\mathbf{P}&=\mathbf{S}_{N}\Big(\mathbf{\Lambda}_{\mathbf{A}_N}+\eta_\alpha (k_{1,+}+k_{1,-})\mathbf{\Lambda}_{n_3} \otimes \mathbf{\Lambda}_{n_2} \otimes \mathbf{\Lambda}_1^{\alpha} + \eta_\beta (k_{2,+}+k_{2,-})\mathbf{\Lambda}_{n_3} \otimes \mathbf{\Lambda}_2^{\beta} \otimes \mathbf{\Lambda}_{n_1}  \\
&~~~~~+ \eta_\gamma (k_{3,+}+k_{3,-}) \mathbf{\Lambda}_3^{\gamma} \otimes \mathbf{\Lambda}_{n_2} \otimes \mathbf{\Lambda}_{n_1}\Big)\mathbf{S}_{N}\\
&=:\mathbf{S}_{N} \mathbf{\Lambda} \mathbf{S}_{N},
\end{aligned}
\end{equation}
where $\mathbf{S}_{N}=\underset{i=3}{\overset{1}{\otimes}}\mathbf{S}_{n_i}$; $\mathbf{\Lambda}_{\mathbf{A}_N}=\underset{i=3}{\overset{1}{\otimes}}\mathbf{\Lambda}_{n_i}$ and $\mathbf{\Lambda}=\mathbf{\Lambda}_{\mathbf{A}_N} + \eta_\alpha (k_{1,+}+k_{1,-})\mathbf{\Lambda}_{n_3} \otimes \mathbf{\Lambda}_{n_2} \otimes \mathbf{\Lambda}_1^{\alpha} + \eta_\beta (k_{2,+}+k_{2,-})\mathbf{\Lambda}_{n_3} \otimes \mathbf{\Lambda}_2^{\beta} \otimes \mathbf{\Lambda}_{n_1} + \eta_\gamma (k_{3,+}+k_{3,-}) \mathbf{\Lambda}_3^{\gamma} \otimes \mathbf{\Lambda}_{n_2} \otimes \mathbf{\Lambda}_{n_1}$ with diagonal matrices $\mathbf{\Lambda}_{n_i}$ and $\mathbf{\Lambda}_i^{\delta_i}$ storing the eigenvalues of the matrices $\mathbf{A}_{n_i}$ and $\tau(\mathcal{H}(\mathbf{T}_{n_i}^{\delta_i}))$, respectively.

Before proving the invertibility of the proposed preconditioner $\mathbf{P}$, we first give two necessary lemmas.

\begin{lemma}\label{tridiag_eig}
    The tridiagonal matrix $\mathbf{A}_{n_i}$ defined in \eqref{tridiag_matrix} is SPD with eigenvalues lying in the interval $(\frac{1}{2}, 1)$.
\end{lemma}
\begin{proof}
The results can be obtained directly by using the Ger\v{s}gorin disk theorem \cite{bai2021matrix,axelsson1996iterative,varga2011gervsgorin}.
\end{proof}


\begin{lemma}\label{1D_tau_SPD}
    For $\delta_i \in (0,1)$,  the matrices $\mathcal{H}(\mathbf{T}_{n_i}^{\delta_i})$ and $\tau(\mathcal{H}(\mathbf{T}_{n_i}^{\delta_i}))$ are SPD.
\end{lemma}

\begin{proof}
In fact, it has been established in \cite{fu2019stability} that $\mathcal{H}(\mathbf{T}_{n_i}^{\delta_i})$ is SPD. In the following, we show that matrix $\tau(\mathcal{H}(\mathbf{T}_{n_i}^{\delta_i}))$ is SPD as well. Let $\lambda_k$ be the $k$-th eigenvalue of $\tau(\mathcal{H}(\mathbf{T}_{n_i}^{\delta_i}))$ for $k=1,2,\ldots,n_i$, then $\lambda_k$ can be determined by using the first column elements of $\mathcal{H}(\mathbf{T}_{n_i}^{\delta_i})$, i.e.,
\begin{equation*}
\begin{aligned}
\lambda_k&= q_1^{(\delta_i)}+ \left(q_0^{(\delta_i)}+q_2^{(\delta_i)}\right)\cos(\frac{k\pi}{n_i+1})+\sum\limits_{j=2}^{{n_i}-1}q_{k+1}^{(\delta_i)}\cos(\frac{jk\pi}{n_i+1})\\
&>q_1^{(\delta_i)}+ \left(q_0^{(\delta_i)}+q_2^{(\delta_i)}\right)+\sum\limits_{j=2}^{{n_i}-1}q_{k+1}^{(\delta_i)}~~~~~~\text{(by~Lemma~\ref{coef_prop1} (i))}\\
&=\sum\limits_{j=0}^{{n_i}-1}q_{k}^{(\delta_i)}\\
&>\sum\limits_{j=0}^{\infty}q_{k}^{(\delta_i)}~~~~~~\text{(by~Lemma~\ref{coef_prop1} (i))}\\
&=0~~~~~~\text{(by~Lemma~\ref{coef_prop1} (ii))}.
\end{aligned}
\end{equation*}
The proof is complete.
\end{proof}

Based on Lemmas \ref{tridiag_eig} and \ref{1D_tau_SPD}, the invertibility of $\mathbf{P}$ can be illustrated in the following lemma. 
\begin{lemma} \label{spd_tau_pre}
	Let $\mathbf{P}$ be the matrix defined in (\ref{pre_definition}), then $\mathbf{P}$ is invertible, and $\lambda_{\min}(\mathbf{P})>\frac{1}{8}$.
\end{lemma}
\begin{proof}
From (\ref{diag_tau_pre}) and Lemmas \ref{tridiag_eig} and \ref{1D_tau_SPD}, it is easy to know that
\begin{equation*}
\begin{aligned}
\mathbf{\Lambda}&=\mathbf{\Lambda}_{\mathbf{A}_N} + \eta_\alpha (k_{1,+}+k_{1,-})\mathbf{\Lambda}_{n_3} \otimes \mathbf{\Lambda}_{n_2} \otimes \mathbf{\Lambda}_1^{\alpha} + \eta_\beta (k_{2,+}+k_{2,-})\mathbf{\Lambda}_{n_3} \otimes \mathbf{\Lambda}_2^{\beta} \otimes \mathbf{\Lambda}_{n_1} \\
&~~~~+ \eta_\gamma (k_{3,+}+k_{3,-}) \mathbf{\Lambda}_3^{\gamma} \otimes \mathbf{\Lambda}_{n_2} \otimes \mathbf{\Lambda}_{n_1}\\
&\succ \mathbf{\Lambda}_{\mathbf{A}_N}\\
&\succ \left(\frac{1}{2}\right)^3 \mathbf{I}_N.
\end{aligned}
\end{equation*}
The result follows.
\end{proof}

From (\ref{stiff_matrix}), we know that when $k_{i,+}=k_{i,-}\,(i=1,2,3)$, $\mathbf{B}_{n_i}^{\delta_i}$ defined in (\ref{stiff_matrix}) is SPD \cite{fu2019stability}, which combined with Lemma \ref{tridiag_eig} indicates that the linear systems (\ref{original_matrix_form}) are SPD. In this case, the preconditioned CG (PCG) method with the preconditioner $\mathbf{P}$ is used to solve the systems (\ref{original_matrix_form}). Otherwise, when $k_{1,+}\not=k_{1,-}$ or $k_{2,+}\not=k_{2,-}$ or $k_{3,+}\not=k_{3,-}$, the matrix $\mathbf{B}_{n_i}^{\delta_i}$ is non-symmetric, then the PGMRES method with the preconditioner $\mathbf{P}$ will be adopted to solve (\ref{original_matrix_form}). In what follows, the convergence behaviors of the two methods will be discussed in detail based on the following important inequality.

\begin{lemma}\cite{lin2018efficient}\label{ineq_lemma}
	For nonnegative numbers $\xi_i$ and positive numbers $\zeta_i$ $(1\leq i\leq m)$, it holds that
\begin{equation*}
\min\limits_{1\leq i\leq m}\frac{\xi_i}{\zeta_i}\leq\bigg(\sum\limits_{i=1}^{m}\zeta_i\bigg)^{-1}\bigg(\sum\limits_{i=1}^{m}\xi_i\bigg)\leq\max\limits_{1\leq i\leq m}\frac{\xi_i}{\zeta_i}.
	\end{equation*}
\end{lemma}

\subsection{Convergence of the CG method for symmetric systems}
It is well known that the convergence rate of the CG method can be described by the eigenvalue distribution of the coefficient matrix. The more clustered the eigenvalues
are, the faster the convergence rate will be. In particular, if the eigenvalues of the preconditioned matrix are bounded with mesh-size independent bounds, then the CG method converges linearly \cite{chan2007introduction,huang2021spectral}. 
In this subsection, we will show the effectiveness of the PCG method for (\ref{original_matrix_form}) in the symmetric case by showing the bounded eigenvalue distribution of the preconditioned matrix $\mathbf{P}^{-1/2}\mathbf{A}\mathbf{P}^{-1/2}$.

\begin{lemma}\cite{huang2021spectral}\label{monotony_coefficients}
Let $T_n=\left[t_{|i-j|}\right]$ be a symmetric Toeplitz matrix and $\tau(T_n)$ be the corresponding $\tau$ matrix. If the entries of $T_n$ are equipped with the following properties,
\begin{equation} \label{bounded_spectrum_lemma}
t_0>0, t_1<t_2<t_3<\cdots<t_{n-1}<0 \text { and } t_0+2 \sum_{i=1}^n t_i>0,
\end{equation}
or
\begin{equation}
t_0<0, t_1>t_2>t_3>\cdots>t_{n-1}>0 \text { and } t_0+2 \sum_{i=1}^n t_i<0,
\end{equation}
the eigenvalues of the matrix $\tau\left(T_n\right)^{-1} T_n$ satisfy
\begin{equation}
1 / 2<\lambda\left(\tau\left(T_n\right)^{-1} T_n\right)<3 / 2 .
\end{equation}
\end{lemma}

\begin{theorem}\label{CG_bounded_sprctrum}
Let $\mathbf{A}, \mathbf{P}$ be the matrices defined in \eqref{linear_sys} and \eqref{pre_definition}, respectively, and $k_{i,+}=k_{i,-}\,(i=1,2,3)$. Then the preconditioned matrix $\mathbf{P}^{-1/2}\mathbf{A} \mathbf{P}^{-1/2}$ is symmetric and the eigenvalues satisfy
\begin{equation}
1/2<\lambda\left(\mathbf{P}^{-1/2}\mathbf{A}\mathbf{P}^{-1/2}\right)<3/2 .
\end{equation}
\end{theorem}
\begin{proof}
When $k_{i,+}=k_{i,-}\,(i=1,2,3)$, from (\ref{stiff_matrix}), we note that $\mathbf{B}_{n_i}^{\delta_i}$ is a symmetric Toeplitz matrix, i.e.,
$$
\mathbf{B}_{n_i}^{\delta_i}=k_{i,+} \mathbf{T}^{\delta_i}_{n_i}+k_{i,-}\left(\mathbf{T}^{\delta_i}_{n_i}\right)^{\top}=(k_{i,+}+k_{i,-})\mathcal{H}(\mathbf{T}^{\delta_i}_{n_i}),
$$ 
where 
the first column of $\mathcal{H}(\mathbf{T}^{\delta_i}_{n_i})$ is $[q_1^{(\delta_i)},\frac{q_0^{(\delta_i)}+q_2^{(\delta_i)}}{2},\frac{q_3^{(\delta_i)}}{2},\ldots,\frac{q_{n_i}^{(\delta_i)}}{2}]^{T}$. Combining Lemmas \ref{coef_prop1}, \ref{coef_prop2} and \ref{monotony_coefficients}, one can immediately conclude that 
\begin{equation} \label{1D_bound}
1/2<\lambda\left(\tau\left(\mathcal{H}(\mathbf{T}^{\delta_i}_{n_i})\right)^{-1} \mathcal{H}(\mathbf{T}^{\delta_i}_{n_i})\right)<3/2.
\end{equation}
Note that for any nonzero vector $\mathbf{y} \in \mathbb{R}^{N\times 1}$,
\begin{eqnarray*}
\frac{\mathbf{y}^T \left(\mathbf{A}_{n_1^{+}} \otimes \mathbf{B}_{n_1}^{\alpha}\right) \mathbf{y}}{\mathbf{y}^T \left(\mathbf{A}_{n_1^{+}} \otimes \tau(\mathcal{H}(\mathbf{B}_{n_1}^{\alpha}))\right) \mathbf{y}} &=&\frac{\mathbf{y}^T \left(\mathbf{A}_{n_3}{\otimes}\mathbf{A}_{n_2} \otimes \mathbf{B}_{n_1}^{\alpha}\right) \mathbf{y}}{\mathbf{y}^T \left(\mathbf{A}_{n_3}{\otimes}\mathbf{A}_{n_2} \otimes \tau(\mathcal{H}(\mathbf{B}_{n_1}^{\alpha}))\right) \mathbf{y}}\\
&=&\frac{\mathbf{y}^T \left(\mathbf{L}_{n_3}\mathbf{L}_{n_3}^T{\otimes}\mathbf{L}_{n_2}\mathbf{L}_{n_2}^T \otimes (k_{1,+}+k_{1,-})\mathcal{H}(\mathbf{T}^{\alpha}_{n_1})\right) \mathbf{y}}{\mathbf{y}^T \left(\mathbf{L}_{n_3}\mathbf{L}_{n_3}^T{\otimes}\mathbf{L}_{n_2}\mathbf{L}_{n_2}^T \otimes (k_{1,+}+k_{1,-}) \tau(\mathcal{H}(\mathbf{T}_{n_1}^{\alpha}))\right) \mathbf{y}},\\
&=&\frac{\mathbf{z}^T \left(\mathbf{I}_{n_3}{\otimes}\mathbf{I}_{n_2} \otimes \mathcal{H}(\mathbf{T}^{\alpha}_{n_1})\right) \mathbf{z}}{\mathbf{z}^T \left(\mathbf{I}_{n_3}{\otimes}\mathbf{I}_{n_2} \otimes \tau(\mathcal{H}(\mathbf{T}^{\alpha}_{n_1}))\right) \mathbf{z}},~~~~\mathbf{z}= \left(\mathbf{L}_{n_3}^T \otimes \mathbf{L}_{n_2}^T \otimes \mathbf{I}_{n_1}\right)\mathbf{y},
\end{eqnarray*}
where $\mathbf{A}_{n_i}=\mathbf{L}_{n_i}\mathbf{L}_{n_i}^T$ for $i=2, 3$ is the Cholesky factorization of the SPD matrix $\mathbf{A}_{n_i}$.
By (\ref{1D_bound}), the Rayleigh quotient theorem \cite{chen2005matrix} and properties of the Kronecker product, we have 
\begin{equation*}
\begin{footnotesize}
\begin{aligned}
    \frac{1}{2}<\lambda_{\min}\left(\tau\left(\mathcal{H}(\mathbf{T}^{\alpha}_{n_1})\right)^{-1} \mathcal{H}(\mathbf{T}^{\alpha}_{n_1})\right)\leq\frac{\mathbf{z}^T \left(\mathbf{I}_{n_3}{\otimes}\mathbf{I}_{n_2}\otimes \mathcal{H}(\mathbf{T}^{\alpha}_{n_1})\right)\mathbf{z}}{\mathbf{z}^T \left(\mathbf{I}_{n_3}{\otimes}\mathbf{I}_{n_2} \otimes \tau(\mathcal{H}(\mathbf{T}^{\alpha}_{n_1}))\right) \mathbf{z}}\leq\lambda_{\max}\left(\tau\left(\mathcal{H}(\mathbf{T}^{\alpha}_{n_1})\right)^{-1} \mathcal{H}(\mathbf{T}^{\alpha}_{n_1})\right)<\frac{3}{2},
\end{aligned}
\end{footnotesize}
\end{equation*}
i.e.,
\begin{equation}\label{Rayleigh_quotient_1}
    \frac{1}{2}<\frac{\mathbf{y}^T \left(\mathbf{A}_{n_1^{+}} \otimes \mathbf{B}_{n_1}^{\alpha}\right) \mathbf{y}}{\mathbf{y}^T \left(\mathbf{A}_{n_1^{+}} \otimes \tau(\mathcal{H}(\mathbf{B}_{n_1}^{\alpha}))\right) \mathbf{y}}< \frac{3}{2}.
\end{equation}
Similarly, it holds that
\begin{equation}\label{Rayleigh_quotient_2}
    \frac{1}{2}<\frac{\mathbf{y}^T \left(\mathbf{A}_{n_3} \otimes \mathbf{B}_{n_2}^{\beta} \otimes \mathbf{A}_{n_1}\right) \mathbf{y}}{\mathbf{y}^T \left(\mathbf{A}_{n_3} \otimes \tau(\mathcal{H}(\mathbf{B}_{n_2}^{\beta})) \otimes \mathbf{A}_{n_1}\right) \mathbf{y}}< \frac{3}{2},
\end{equation}
and
\begin{equation}\label{Rayleigh_quotient_3}
    \frac{1}{2}<\frac{\mathbf{y}^T \left(\mathbf{B}_{n_3}^{\gamma} \otimes \mathbf{A}_{n_3^{-}}\right) \mathbf{y}}{\mathbf{y}^T \left(\tau(\mathcal{H}(\mathbf{B}_{n_3}^{\gamma}) \otimes \mathbf{A}_{n_3^{-}}\right) \mathbf{y}}< \frac{3}{2}.
\end{equation}
Based on (\ref{Rayleigh_quotient_1})-(\ref{Rayleigh_quotient_3}), we have
\begin{equation}\label{}
\begin{small}
\begin{aligned}
&\frac{\mathbf{y}^T \left(\mathbf{A}_N +\frac{1}{2}\eta_\alpha \mathbf{A}_{n_1^{+}} \otimes \tau(\mathcal{H}(\mathbf{B}_{n_1}^{\alpha}))+\frac{1}{2}\eta_\beta \mathbf{A}_{n_3} \otimes \tau(\mathcal{H}(\mathbf{B}_{n_2}^{\beta})) \otimes \mathbf{A}_{n_1}+\frac{1}{2}\eta_\gamma \tau(\mathcal{H}(\mathbf{B}_{n_3}^{\gamma})) \otimes \mathbf{A}_{n_3^{-}}\right) \mathbf{y}}{\mathbf{y}^T \left(\mathbf{A}_N +\eta_\alpha \mathbf{A}_{n_1^{+}} \otimes \tau(\mathcal{H}(\mathbf{B}_{n_1}^{\alpha}))+\eta_\beta \mathbf{A}_{n_3} \otimes \tau(\mathcal{H}(\mathbf{B}_{n_2}^{\beta})) \otimes \mathbf{A}_{n_1}+\eta_\gamma \tau(\mathcal{H}(\mathbf{B}_{n_3}^{\gamma})) \otimes \mathbf{A}_{n_3^{-}}\right) \mathbf{y}}\\
&<\frac{\mathbf{y}^T \left(\mathbf{A}_N +\eta_\alpha \mathbf{A}_{n_1^{+}} \otimes \mathbf{B}_{n_1}^{\alpha}+\eta_\beta \mathbf{A}_{n_3} \otimes \mathbf{B}_{n_2}^{\beta} \otimes \mathbf{A}_{n_1}+\eta_\gamma \mathbf{B}_{n_3}^{\gamma} \otimes \mathbf{A}_{n_3^{-}}\right) \mathbf{y}}{\mathbf{y}^T \left(\mathbf{A}_N +\eta_\alpha \mathbf{A}_{n_1^{+}} \otimes \tau(\mathcal{H}(\mathbf{B}_{n_1}^{\alpha}))+\eta_\beta \mathbf{A}_{n_3} \otimes \tau(\mathcal{H}(\mathbf{B}_{n_2}^{\beta})) \otimes \mathbf{A}_{n_1}+\eta_\gamma \tau(\mathcal{H}(\mathbf{B}_{n_3}^{\gamma})) \otimes \mathbf{A}_{n_3^{-}}\right) \mathbf{y}}\\
&<\frac{\mathbf{y}^T \left(\mathbf{A}_N +\frac{3}{2}\eta_\alpha \mathbf{A}_{n_1^{+}} \otimes \tau(\mathcal{H}(\mathbf{B}_{n_1}^{\alpha}))+\frac{3}{2}\eta_\beta \mathbf{A}_{n_3} \otimes \tau(\mathcal{H}(\mathbf{B}_{n_2}^{\beta})) \otimes \mathbf{A}_{n_1}+\frac{3}{2}\eta_\gamma \tau(\mathcal{H}(\mathbf{B}_{n_3}^{\gamma})) \otimes \mathbf{A}_{n_3^{-}}\right) \mathbf{y}}{\mathbf{y}^T \left(\mathbf{A}_N +\eta_\alpha  \mathbf{A}_{n_1^{+}} \otimes \tau(\mathcal{H}(\mathbf{B}_{n_1}^{\alpha}))+\eta_\beta \mathbf{A}_{n_3} \otimes \tau(\mathcal{H}(\mathbf{B}_{n_2}^{\beta})) \otimes \mathbf{A}_{n_1}+\eta_\gamma \tau(\mathcal{H}(\mathbf{B}_{n_3}^{\gamma})) \otimes \mathbf{A}_{n_3^{-}}\right) \mathbf{y}},
\end{aligned}
\end{small}
\end{equation}
which combining with Lemma \ref{ineq_lemma} implies
\begin{eqnarray*}
\lambda_{\min} (\mathbf{P}^{-1/2}\mathbf{A}\mathbf{P}^{-1/2})=\min\limits_{\mathbf{y}}\frac{\mathbf{y}^T \mathbf{A} \mathbf{y}}{\mathbf{y}^T \mathbf{P} \mathbf{y}}=\min\{1,\frac{1}{2}\}&=& \frac{1}{2},\\
\lambda_{\max} (\mathbf{P}^{-1/2}\mathbf{A}\mathbf{P}^{-1/2})=\max\limits_{\mathbf{y}}\frac{\mathbf{y}^T \mathbf{A} \mathbf{y}}{\mathbf{y}^T \mathbf{P} \mathbf{y}}=\max\{1,\frac{3}{2}\}&=& \frac{3}{2}.
\end{eqnarray*}
The proof is complete.
\end{proof}

Theorem \ref{CG_bounded_sprctrum} illustrates that the eigenvalues of the preconditioned matrix are uniformly bounded by two positive constants independent of matrix size, implying the PCG method with the proposed preconditioner $\mathbf{P}$ for solving \eqref{linear_sys} converges linearly.

\subsection{Convergence of the GMRES method for non-symmetric systems}
When $k_{1,+}\not=k_{1,-}$ or $k_{2,+}\not=k_{2,-}$ or $k_{3,+}\not=k_{3,-}$, systems (\ref{original_matrix_form}) are non-symmetric. Therefore, we use the GMRES method to solve the following one-sided preconditioned systems
\begin{equation}\label{one_side_system}
\mathbf{P}^{-1}\mathbf{A} \mathbf{u}^{m}=\mathbf{P}^{-1} \mathbf{b}^{m-1}, \quad 1 \leq m \leq M.
\end{equation}
Unlike the CG method, the behavior of the GMRES cannot be determined from eigenvalues alone \cite{greenbaum1997iterative}. However, residual-norm estimates provide practical ways to characterize its convergence; see, e.g., \cite{lin2023tau,lin2021all,liu2020fast,qin2023sine}.

In order to show the efficiency of the proposed preconditioner $\mathbf{P}$, we first introduce the following auxiliary two-sided preconditioned systems
\begin{align}\label{two_side_system}
& \mathbf{P}^{-\frac{1}{2}}  \mathbf{A} \underbrace{\mathbf{P}^{-\frac{1}{2}} \mathbf{\hat{u}}^{m}}_{=:\mathbf{u}^{m}}=\mathbf{P}^{-\frac{1}{2}} \mathbf{b}^{m-1}, \quad 1 \leq m \leq M.
\end{align}

\begin{lemma}\cite[Proposition 7.3]{elman2014finite}\label{lemma:gmres}
    Let $\mathbf{Z}\mathbf{v}=\mathbf{w}$ be a real square linear system with $\mathcal{H}(\mathbf{Z}) \succ \mathcal{O}$. Then, the residuals of the iterates generated by applying (restarted or non-restarted) GMRES method with an arbitrary initial guess to solve $\mathbf{Z}\mathbf{v}=\mathbf{w}$ satisfy
    \[
        \| \mathbf{r}_k \|_2 \leq \left( 1 - \frac{\lambda_{\min}( \mathcal{H}(\mathbf{Z})  )^2 }{ \lambda_{\min}( \mathcal{H}(\mathbf{Z})  ) \lambda_{\max}( \mathcal{H}(\mathbf{Z})  ) + \rho( \mathcal{S}(\mathbf{Z})   )^2   }   \right)^{k/2} \| \mathbf{r}_0 \|_2,
    \]
    where $\mathbf{r}_k = \mathbf{w} - \mathbf{Z}\mathbf{v}_k$ is the residual vector at the $k$-th GMRES iteration with $\mathbf{v}_k$ ($k \geq 1$) being the corresponding iterative solution.
\end{lemma}

In what follows, we will evaluate the lower and upper bounds of the eigenvalues of $\mathcal{H}(\mathbf{P}^{-\frac{1}{2}}  \mathbf{A} \mathbf{P}^{-\frac{1}{2}})$ and the upper bound of the spectral radius of $\mathcal{S}(\mathbf{P}^{-\frac{1}{2}}  \mathbf{A} \mathbf{P}^{-\frac{1}{2}})$ one by one.

\begin{lemma}\label{prop:eigen_S}
Let $\mathbf{A}, \mathbf{P}$ be the matrices defined in \eqref{linear_sys} and \eqref{pre_definition}, respectively. Then, the eigenvalues of $\mathbf{P}^{-1/2} \mathcal{H}(\mathbf{A}) \mathbf{P}^{-1/2}$ lie in $(\frac{1}{2},\frac{3}{2})$.
\end{lemma}
\begin{proof}
Let $(\lambda, \mathbf{w})$ be an arbitrary eigenpair of $\mathbf{P}^{-\frac{1}{2}}   \mathcal{H}(\mathbf{A})  \mathbf{P}^{-\frac{1}{2}}$. Then, it holds that
\begin{eqnarray*}
    \lambda &=& \frac{ \mathbf{w}^T    \mathbf{P}^{-\frac{1}{2}}   \mathcal{H}(\mathbf{A})  \mathbf{P}^{-\frac{1}{2}} \mathbf{w} }{\mathbf{w}^T \mathbf{w}} \\
    &=& \frac{ \mathbf{y}^T    \mathcal{H}(\mathbf{A}) \mathbf{y} }{\mathbf{y}^T \mathbf{P} \mathbf{y}},\,\, (\mathbf{y}=\mathbf{P}^{-\frac{1}{2}} \mathbf{w})\\
    &=&\frac{ \mathbf{y}^T    \left(\mathbf{A}_N +\eta_\alpha \mathbf{A}_{n_1^{+}} \otimes \mathcal{H}(\mathbf{B}_{n_1}^{\alpha})+\eta_\beta \mathbf{A}_{n_3} \otimes \mathcal{H}(\mathbf{B}_{n_2}^{\beta}) \otimes \mathbf{A}_{n_1}+\eta_\gamma \mathcal{H}(\mathbf{B}_{n_3}^{\gamma}) \otimes \mathbf{A}_{n_3^{-}}\right) \mathbf{y} }{ \mathbf{y}^T    \left(\mathbf{A}_N +\eta_\alpha \mathbf{A}_{n_1^{+}} \otimes \tau(\mathcal{H}(\mathbf{B}_{n_1}^{\alpha}))+\eta_\beta \mathbf{A}_{n_3} \otimes \tau(\mathcal{H}(\mathbf{B}_{n_2}^{\beta})) \otimes \mathbf{A}_{n_1}+\eta_\gamma \tau(\mathcal{H}(\mathbf{B}_{n_3}^{\gamma})) \otimes \mathbf{A}_{n_3^{-}}\right) \mathbf{y}}.
\end{eqnarray*}
Now, combining (\ref{Rayleigh_quotient_1})-(\ref{Rayleigh_quotient_3}) with Lemma \ref{ineq_lemma}, we have
\begin{equation*}
    \frac{1}{2}=\min\{1,\frac{1}{2}\}< \frac{ \mathbf{w}^T    \mathbf{P}^{-\frac{1}{2}}   \mathcal{H}(\mathbf{A})  \mathbf{P}^{-\frac{1}{2}} \mathbf{w} }{\mathbf{w}^T \mathbf{w}} < \max\{1,\frac{3}{2}\}=\frac{3}{2},
\end{equation*}
which, combined with the Rayleigh quotient theorem \cite{chen2005matrix} for real symmetric matrix implies 
\begin{equation*}
    \frac{1}{2}< \lambda(\mathbf{P}^{-1/2} \mathcal{H}(\mathbf{A}) \mathbf{P}^{-1/2}) < \frac{3}{2}.
\end{equation*}
The proof is complete.
\end{proof}
To estimate the upper bound of $\rho(\mathcal{S}(\mathbf{P}^{-\frac{1}{2}}  \mathbf{A} \mathbf{P}^{-\frac{1}{2}}))$, the relationship between the real and imaginary parts of the generating function for $\mathbf{T}_{n_i}^{\delta_i}$ is needed, which is presented in the following lemma.

\begin{lemma}\label{non_symmetric_ratio}
Let $g_{\delta_i}( \theta )$ be the generating function of the Toeplitz matrix $\mathbf{T}_{n_i}^{\delta_i}$ defined in Equation (\ref{T_ni}). For $\delta_i\in(0,1)$ and $\forall \theta \in \mathbb{R} \backslash\{2k\pi \mid k \in \mathbb{Z}\}$, we have, 
\begin{description}
  \item[(i)]  $\operatorname{Re}(g_{\delta_i}( \theta ))>0$;
  \item[(ii)]  $\frac{\left|\operatorname{Im}(g_{\delta_i}( \theta ))\right|}{\operatorname{Re}(g_{\delta_i}( \theta ))}<\tan{\left( \frac{\delta_i}{2} \pi \right)}$.
\end{description}
\end{lemma}

\begin{proof}
For $\delta_i\in \mathbbm{C}$ with $\textrm{Re}(\delta_i)>0$, it has been shown that $|q_i^{(\delta_i)}|=\mathcal{O}(i^{-2-\textrm{Re}(\delta_i)})$, see Lemma 3.4 in \cite{donatelli2018spectral}, then $g_{\delta_i}( \theta )$ is analytic on $\{s\in\mathbbm{C}\big|\textrm{Re}(s)>0\}$ with respect to $\delta_i$. Hence with the analytic continuation of the Lerch transcendent \cite{bateman1953higher} $\Phi(z,s,\nu)=\sum_{k=0}^{\infty}(k+\nu)^{-s}z^{k}$ that is
\[\Phi\vert_{S_{1}}(z,s,\nu)=z^{-\nu}\Gamma(1-s)\sum\limits_{n=-\infty}^{\infty}(-\log z+2n\pi \mathbf{i})^{s-1}e^{2n\pi \mathbf{i}\nu},\]
for $z\in\mathbb{C}\setminus[1,\infty),~\lvert\arg(-\log z+2n\pi \mathbf{i})\rvert\leq\pi,~0<\nu\leq 1,~s\in S_{1}=\{s\in\mathbb{C}\big\vert\textrm{Re}~s<0\}$, we have for $\delta_i\in\{s\in\mathbb{C}\big\vert 0<\textrm{Re}~(s)<1\},~\theta\in(0,\pi]$,
\begin{align*}
g_{\delta_i}( \theta )&=\sum_{j=0}^{\infty}q_j^{(\delta_i)} e^{\mathbf{i} (j-1) \theta}\\
&=-e^{-\mathbf{i}\theta}\Phi(e^{\mathbf{i}\theta},-\delta_i,\frac{1}{2})+3\Phi(e^{\mathbf{i}\theta},-\delta_i,\frac{1}{2})-3e^{\mathbf{i}\theta}\Phi(e^{\mathbf{i}\theta},-\delta_i,\frac{1}{2})+e^{2\mathbf{i}\theta}\Phi(e^{\mathbf{i}\theta},-\delta_i,\frac{1}{2})\\
&=\left(e^{2\mathbf{i}\theta}-3e^{\mathbf{i}\theta}-e^{-\mathbf{i}\theta}+3\right)\Phi(e^{\mathbf{i}\theta},-\delta_i,\frac{1}{2})\\
&=\left(e^{2\mathbf{i}\theta}-3e^{\mathbf{i}\theta}-e^{-\mathbf{i}\theta}+3\right) \left((e^{\mathbf{i}\theta})^{\frac{1}{2}}\Gamma(1+\delta_i)\sum\limits_{n=-\infty}^{\infty}(-\mathbf{i}\theta+2n\pi\mathbf{i})^{-\delta_i-1} e^{2n\pi\mathbf{i}\frac{1}{2}}\right)     \\ 
&=\left(e^{\frac{3}{2}\mathbf{i}\theta}-3e^{\frac{1}{2}\mathbf{i}\theta}-e^{-\frac{3}{2}\mathbf{i}\theta}+3e^{-\frac{1}{2}\mathbf{i}\theta}\right) \left(\Gamma(1+\delta_i)\sum\limits_{n=-\infty}^{\infty}(-1)^n(-\mathbf{i}\theta+2n\pi\mathbf{i})^{-\delta_i-1} \right)
\end{align*}
Note that 
$$
e^{\frac{3}{2}\mathbf{i}\theta}-3e^{\frac{1}{2}\mathbf{i}\theta}-e^{-\frac{3}{2}\mathbf{i}\theta}+3e^{-\frac{1}{2}\mathbf{i}\theta}=\left(2\sin(\frac{3}{2}\theta)-6\sin(\frac{1}{2}\theta)\right)\mathbf{i},
$$ 
which is a pure imaginary number. For $\sum\limits_{n=-\infty}^{\infty}(-1)^n(-\mathbf{i}\theta+2n\pi\mathbf{i})^{-\delta_i-1}$, we have
\begin{align*}
&\sum\limits_{n=-\infty}^{\infty}(-1)^n(-\mathbf{i}\theta+2n\pi\mathbf{i})^{-\delta_i-1}\\
=&\sum\limits_{n=1}^{\infty}(-1)^n(-\mathbf{i}\theta+2n\pi\mathbf{i})^{-\delta_i-1} + \sum\limits_{n=0}^{\infty}(-1)^{-n}(-\mathbf{i}\theta-2n\pi\mathbf{i})^{-\delta_i-1}\\
=&\sum\limits_{n=0}^{\infty}(-1)^{n+1}(-\mathbf{i}\theta+2(n+1)\pi\mathbf{i})^{-\delta_i-1} + \sum\limits_{n=0}^{\infty}(-1)^{-n}(-\mathbf{i}\theta-2 n\pi\mathbf{i})^{-\delta_i-1}\\
=&\sum\limits_{n=0}^{\infty}(-1)^{n+1}(-\theta+2(n+1)\pi)^{-\delta_i-1}e^{\frac{-\delta_i-1}{2}\pi\mathbf{i}} + \sum\limits_{n=0}^{\infty}(-1)^{n}(\theta+2n\pi)^{-\delta_i-1}e^{\frac{\delta_i+1}{2}\pi\mathbf{i}}\\
=&\sum\limits_{n=0}^{\infty}(-1)^n\big[(-\theta+2(n+1)\pi)^{-\delta_i-1}-(\theta+2n\pi)^{-\delta_i-1}\big]\sin(\frac{\delta_i}{2}\pi) \\
&\hspace{20pt} +\mathbf{i}\sum\limits_{n=0}^{\infty}(-1)^n\big[(-\theta+2(n+1)\pi)^{-\delta_i-1}+(\theta+2n\pi)^{-\delta_i-1}\big]\cos(\frac{\delta_i}{2}\pi).
\end{align*}

Therefore, for $\delta_i\in(0,1)$ and $\forall\theta\in(0,\pi]$, we have 
\begin{equation*}
\begin{aligned}
&\operatorname{Re}(g_{\delta_i}(\theta ))\\
=&-\Gamma(1+\delta_i)\cos(\frac{\delta_i}{2}\pi)\left(2\sin(\frac{3}{2}\theta)-6\sin(\frac{1}{2}\theta)\right)\sum\limits_{n=0}^{\infty}(-1)^n\big[(-\theta+2(n+1)\pi)^{-\delta_i-1}+ (\theta+2n\pi)^{-\delta_i-1}\big].
\end{aligned}
\end{equation*}
Note that for $\forall\theta\in(0,\pi]$, it is obvious that $\Gamma(1+\delta_i)>0$, $\cos(\frac{\delta_i}{2}\pi)>0$. Moreover, for $\theta\in(0,\pi)$, we have
\begin{eqnarray}
 2\sin(\frac{3}{2}\theta)-6\sin(\frac{1}{2}\theta)&=&2\sin(\theta+\frac{\theta}{2})-6\sin(\theta-\frac{\theta}{2})\nonumber\\
 &=&2\sin(\theta)\cos(\frac{\theta}{2})+2\cos(\theta)\sin(\frac{\theta}{2})-6\left(\sin(\theta)\cos(\frac{\theta}{2})-\cos(\theta)\sin(\frac{\theta}{2})\right)\nonumber\\
 &=&4\left(2\sin(\frac{\theta}{2})\cos(\theta)-\sin(\theta)\cos(\frac{\theta}{2})\right)\nonumber\\
 &=&4\left(2\sin(\frac{\theta}{2})\cos(\frac{\theta}{2})\frac{\cos(\theta)}{\cos(\frac{\theta}{2})}-\sin(\theta)\cos(\frac{\theta}{2})\right)\,\Big(\cos(\frac{\theta}{2})\ne 0 ~\text{as}~\theta\in(0,\pi)\Big)\nonumber\\
 &=&4\left(\sin(\theta)\frac{2\cos^2(\frac{\theta}{2})-1}{\cos(\frac{\theta}{2})}-\sin(\theta)\cos(\frac{\theta}{2})\right)\nonumber\\
 &=&4\left(2\sin(\theta)\cos(\frac{\theta}{2})-\sin(\theta)\sec(\frac{\theta}{2})-\sin(\theta)\cos(\frac{\theta}{2})\right)\nonumber\\
 &=&4\sin(\theta)\left(\cos(\frac{\theta}{2})-\sec(\frac{\theta}{2})\right)\nonumber\\
 &<&0.\nonumber
\end{eqnarray}
In particular, when $\theta=\pi$,
\begin{equation}\nonumber
2\sin(\frac{3}{2}\theta)-6\sin(\frac{1}{2}\theta)=2\sin(\frac{3\pi}{2})-6\sin(\frac{\pi}{2})=-2-6=-8<0.
\end{equation}
As a result, for $\forall\theta\in(0,\pi]$, we obtain, 
$$
2\sin(\frac{3}{2}\theta)-6\sin(\frac{1}{2}\theta)<0.
$$
At last, since both $\sum\limits_{n=0}^{\infty}(-1)^n(-\theta+2(n+1)\pi)^{-\delta_i-1}$ and $\sum\limits_{n=0}^{\infty}(-1)^n(\theta+2n\pi)^{-\delta_i-1}$ are alternating series, one can easily check that they converge by the Leibniz criterion. Therefore, we can denote the sum of these two alternating series as $A_1$ and $A_2$, respectively. A straightforward calculation implies that 
\begin{eqnarray}
A_1&:=&\sum\limits_{n=0}^{\infty}(-1)^n\big(-\theta+2(n+1)\pi)^{-\delta_i-1}\nonumber\\
&=&\left[(\frac{1}{2\pi-\theta})^{\delta_i+1}-(\frac{1}{4\pi-\theta})^{\delta_i+1}\right]+\left[(\frac{1}{6\pi-\theta})^{\delta_i+1}-(\frac{1}{8\pi-\theta})^{\delta_i+1}\right]+\cdots+\nonumber\\
&&\left[(\frac{1}{2n\pi-\theta})^{\delta_i+1}-(\frac{1}{2n\pi+2\pi-\theta})^{\delta_i+1}\right]+\cdots\nonumber\\
&>&0.\nonumber
\end{eqnarray}
By an analogous argument, we conclude that 
$$
A_2:=\sum\limits_{n=0}^{\infty}(-1)^n(\theta+2n\pi)^{-\delta_i-1}>0.
$$
As a result, we get
$$
A_1+A_2=\sum\limits_{n=0}^{\infty}(-1)^n\big[(-\theta+2(n+1)\pi)^{-\delta_i-1}+ (\theta+2 n\pi)^{-\delta_i-1}\big]>0,
$$ 
which indicates that $\operatorname{Re}(g_{\delta_i}( \theta ))>0$.

Moreover, for $({\bf ii})$, we consider 
\begin{equation*}
\begin{aligned}
&\frac{\left|\operatorname{Im}(g_{\delta_i}( \theta ))\right|}{\operatorname{Re}(g_{\delta_i}( \theta ))}\\
=&\frac{\left|\Gamma(1+\delta_i)\sin(\frac{\delta_i}{2}\pi)\left(2\sin(\frac{3}{2}\theta)-6\sin(\frac{1}{2}\theta)\right)\sum\limits_{n=0}^{\infty}(-1)^n\big[(-\theta+2(n+1)\pi)^{-\delta_i-1}- (\theta+2n\pi)^{-\delta_i-1}\big]\right|}{-\Gamma(1+\delta_i)\cos(\frac{\delta_i}{2}\pi)\left(2\sin(\frac{3}{2}\theta)-6\sin(\frac{1}{2}\theta)\right)\sum\limits_{n=0}^{\infty}(-1)^n\big[(-\theta+2(n+1)\pi)^{-\delta_i-1}+ (\theta+2n\pi)^{-\delta_i-1}\big]}\\
=&\dfrac{\left|\sum\limits_{n=0}^{\infty}(-1)^n\big[(-\theta+2(n+1)\pi)^{-\delta_i-1}- (\theta+2n\pi)^{-\delta_i-1}\big]\right|}{\sum\limits_{n=0}^{\infty}(-1)^n\big[(-\theta+2(n+1)\pi)^{-\delta_i-1}+ (\theta+2n\pi)^{-\delta_i-1}\big]}\tan{\left(\frac{\delta_i}{2} \pi \right)}\\
=:&\frac{|A_1-A_2|}{A_1+A_2}\tan{\left( \frac{\delta_i}{2} \pi \right)},\\
<&\frac{|A_1|+|A_2|}{A_1+A_2}\tan{\left( \frac{\delta_i}{2} \pi \right)},\\
=&\tan{\left( \frac{\delta_i}{2} \pi \right)},
\end{aligned}
\end{equation*}
where the penultimate inequality holds by the triangle inequality for $A_1,A_2>0$, which completes the proof.

Similarly, we have for $\forall\theta\in[-\pi,0)$,
\begin{equation*}
\operatorname{Re}(g_{\delta_i}( \theta ))=\operatorname{Re}(g_{\delta_i}( -\theta ))>0
\end{equation*}
and
\[\frac{\left|\operatorname{Im}(g_{\delta_i}( \theta ))\right|}{\operatorname{Re}(g_{\delta_i}( \theta ))}=\dfrac{\left|-\textrm{Im}(g_{\delta_i}(-\theta ))\right|}{\textrm{Re}(g_{\delta_i}(-\theta ))}<\tan{\left( \frac{\delta_i}{2} \pi \right)}.\]
Hence, by periodicity, we have for $\forall \theta \in \mathbb{R} \backslash\{2 k \pi \mid k \in \mathbb{Z}\}$
\begin{equation*}
\operatorname{Re}(g_{\delta_i}( \theta))>0 \quad \text {and}\quad \frac{\left|\operatorname{Im}(g_{\delta_i}( \theta ))\right|}{\operatorname{Re}(g_{\delta_i}( \theta ))}<\tan{\left( \frac{\delta_i}{2} \pi \right)}.
\end{equation*}
 \end{proof}

\begin{remark}
{\rm
From the proof of Lemma \ref{non_symmetric_ratio}, we note that the explicit expression of generating function $g_{\delta_i}( \theta )$ for the Toeplitz matrix $\mathbf{T}_{n_i}^{\delta_i}$ agrees with the result in \cite{chou2021finite}. However, the proofs of two results are quite different, and our method appears to be simpler. Moreover, Lemma \ref{non_symmetric_ratio} also provides the results for $\operatorname{Re}(g_{\delta_i}( \theta ))$, $\operatorname{Im}(g_{\delta_i}( \theta ))$ and $\frac{\left|\operatorname{Im}(g_{\delta_i}( \theta ))\right|}{\operatorname{Re}(g_{\delta_i}( \theta ))}$, which are not included in \cite{chou2021finite}.
}
\end{remark}
Based on Lemma \ref{non_symmetric_ratio}, and by applying the same argument as in Corollary 3.1 of \cite{lin2023tau}, we can readily derive the following lemma.
\begin{lemma}\label{Her_Skew_Her}
	For any $\delta_i \in (0,1)$, any $n_i \in \mathbb{N}^{+}$ and any $v \in \mathbb{C}^{n_i \times 1}$, it holds that 
\begin{equation}
\left|v^* \mathcal{S}(\mathbf{T}^{\delta_i}_{n_i}) v \right| \leqslant \tan{\left( \frac{\delta_i}{2} \pi \right)} v^* \mathcal{H}(\mathbf{T}^{\delta_i}_{n_i}) v.
\end{equation}
\end{lemma}
With the relationship between the quadratic forms of $\mathcal{S}(\mathbf{T}^{\delta_i}_{n_i})$ and $\mathcal{H}(\mathbf{T}^{\delta_i}_{n_i})$ given in Lemma \ref{Her_Skew_Her}, the upper bound of $\rho(\mathcal{S}(\mathbf{P}^{-\frac{1}{2}}  \mathbf{A} \mathbf{P}^{-\frac{1}{2}}))$ can be obtained.
\begin{lemma}\label{lemma:eigen_Skew} 
Let $\mathbf{A}, \mathbf{P}$ be the matrices defined in \eqref{linear_sys} and \eqref{pre_definition}, respectively. Then,
\begin{equation*}
    \rho\left( \mathbf{P}^{-\frac{1}{2}} \mathcal{S}\left(  \mathbf{A}   \right) \mathbf{P}^{-\frac{1}{2}} \right) \leq \varsigma,    
\end{equation*}
where
\begin{equation*}
    \varsigma= \frac{3}{2} \max _{i \in 1 \wedge 3} \tan{\left( \frac{\delta_i}{2} \pi \right)}\frac{\left| k_{i,+}-k_{i,-} \right|}{ k_{i,+}+k_{i,-} } \geqslant 0.    
\end{equation*}

\end{lemma}

\begin{proof}
     Let $(\lambda, \mathbf{w})$ be an arbitrary eigenpair of $\mathbf{P}^{-\frac{1}{2}} \mathcal{S}\left(  \mathbf{A}   \right) \mathbf{P}^{-\frac{1}{2}}$. Then, we have 

     \begin{align}\label{skew_eig} \nonumber
        |\lambda| &= \left| \frac{ \mathbf{w}^*    \mathcal{S}(\mathbf{A}) \mathbf{w} }{\mathbf{w}^* \mathbf{P} \mathbf{w}} \right|\\  \nonumber
        &=\frac{\Big| \mathbf{w}^*    \Big(\eta_\alpha (k_{1,+}-k_{1,-}) \mathbf{A}_{n_1^{+}} \otimes \mathcal{S}(\mathbf{T}^{\alpha}_{n_1})+\eta_\beta (k_{2,+}-k_{2,-}) \mathbf{A}_{n_3} \otimes \mathcal{S}(\mathbf{T}^{\beta}_{n_2}) \otimes \mathbf{A}_{n_1}}{ \mathbf{w}^*    \Big(\mathbf{A}_N +\eta_\alpha (k_{1,+}+k_{1,-}) \mathbf{A}_{n_1^{+}} \otimes \tau(\mathcal{H}(\mathbf{T}_{n_1}^{\alpha}))+\eta_\beta (k_{2,+}+k_{2,-}) \mathbf{A}_{n_3} \otimes \tau(\mathcal{H}(\mathbf{T}_{n_2}^{\beta})) \otimes \mathbf{A}_{n_1}  }\\    \nonumber
        & ~~~~~\frac{+\eta_\gamma (k_{3,+}-k_{3,-})\mathcal{S}(\mathbf{T}^{\gamma}_{n_3}) \otimes \mathbf{A}_{n_3^{-}}\Big) \mathbf{w} \Big|}{ + \eta_\gamma (k_{3,+}+k_{3,-}) \tau(\mathcal{H}(\mathbf{T}_{n_3}^{\gamma})) \otimes \mathbf{A}_{n_3^{-}} \Big) \mathbf{w} }\\    \nonumber
        &\leq\frac{\Big| \mathbf{w}^*    \Big(\eta_\alpha (k_{1,+}-k_{1,-}) \mathbf{A}_{n_1^{+}} \otimes \mathcal{S}(\mathbf{T}^{\alpha}_{n_1})+\eta_\beta (k_{2,+}-k_{2,-}) \mathbf{A}_{n_3} \otimes \mathcal{S}(\mathbf{T}^{\beta}_{n_2}) \otimes \mathbf{A}_{n_1}}{ \mathbf{w}^*   \Big(\eta_\alpha (k_{1,+}+k_{1,-}) \mathbf{A}_{n_1^{+}} \otimes \tau(\mathcal{H}(\mathbf{T}_{n_1}^{\alpha}))+\eta_\beta (k_{2,+}+k_{2,-}) \mathbf{A}_{n_3} \otimes \tau(\mathcal{H}(\mathbf{T}_{n_2}^{\beta})) \otimes \mathbf{A}_{n_1}  }\\  \nonumber
        & ~~~~~\frac{+\eta_\gamma (k_{3,+}-k_{3,-})\mathcal{S}(\mathbf{T}^{\gamma}_{n_3}) \otimes \mathbf{A}_{n_3^{-}}\Big) \mathbf{w} \Big|}{ + \eta_\gamma (k_{3,+}+k_{3,-}) \tau(\mathcal{H}(\mathbf{T}_{n_3}^{\gamma})) \otimes \mathbf{A}_{n_3^{-}} \Big) \mathbf{w} }\\     \nonumber
        &\leq\frac{\eta_\alpha \left| k_{1,+}-k_{1,-} \right|\left| \mathbf{w}^* \left( \mathbf{A}_{n_1^{+}} \otimes \mathcal{S}(\mathbf{T}^{\alpha}_{n_1})\right) \mathbf{w} \right| + \eta_\beta \left| k_{2,+}-k_{2,-} \right|\left| \mathbf{w}^* \left(\mathbf{A}_{n_3} \otimes \mathcal{S}(\mathbf{T}^{\beta}_{n_2}) \otimes \mathbf{A}_{n_1}\right) \mathbf{w} \right|}{\eta_\alpha (k_{1,+}+k_{1,-}) \mathbf{w}^* \left( \mathbf{A}_{n_1^{+}} \otimes \tau(\mathcal{H}(\mathbf{T}^{\alpha}_{n_1}))\right) \mathbf{w} + \eta_\beta (k_{2,+}+k_{2,-}) \mathbf{w}^* \left(\mathbf{A}_{n_3} \otimes \tau(\mathcal{H}(\mathbf{T}^{\beta}_{n_2})) \otimes \mathbf{A}_{n_1}\right) \mathbf{w} }\\
        &~~~~~\frac{+ \eta_\gamma \left| k_{3,+}-k_{3,-} \right|\left| \mathbf{w}^* \left(\mathcal{S}(\mathbf{T}^{\gamma}_{n_3}) \otimes \mathbf{A}_{n_3^{-}}\right) \mathbf{w} \right|}{+ \eta_\gamma (k_{3,+}+k_{3,-}) \mathbf{w}^* \left(\tau(\mathcal{H}(\mathbf{T}^{\gamma}_{n_3})) \otimes \mathbf{A}_{n_3^{-}}\right) \mathbf{w} }.  
        \end{align}  
        
Note that



\begin{eqnarray*}
         \frac{\left|\mathbf{w}^*     \left(\mathbf{A}_{n_1^{+}} \otimes \mathcal{S}(\mathbf{T}^{\alpha}_{n_1}) \right) \mathbf{w} \right|}{\mathbf{w}^*     \left(\mathbf{A}_{n_1^{+}} \otimes \tau((\mathcal{H}(\mathbf{T}_{n_1}^{\alpha})) \right) \mathbf{w}} 
         &=& \frac{ \left|\mathbf{w}^*     \left(\mathbf{A}_{n_3} \otimes  \mathbf{A}_{n_2} \otimes \mathcal{S}(\mathbf{T}^{\alpha}_{n_1}) \right) \mathbf{w} \right|}{\mathbf{w}^*     \left(\mathbf{A}_{n_3} \otimes  \mathbf{A}_{n_2} \otimes \tau(\mathcal{H}(\mathbf{T}_{n_1}^{\alpha})) \right) \mathbf{w} }\\
        &=&\frac{ \left|\mathbf{w}^*     \left(\mathbf{L}_{n_3}\mathbf{L}_{n_3}^T \otimes  \mathbf{L}_{n_{2}}\mathbf{L}_{n_{2}}^T \otimes \mathcal{S}(\mathbf{T}^{\alpha}_{n_1}) \right) \mathbf{w} \right|}{\mathbf{w}^*     \left(\mathbf{L}_{n_3}\mathbf{L}_{n_3}^T \otimes \mathbf{L}_{n_{2}}\mathbf{L}_{n_{2}}^T \otimes \tau(\mathcal{H}(\mathbf{T}_{n_1}^{\alpha})) \right) \mathbf{w}}\\
        &=&\frac{ \left|\mathbf{z}^*     \left(\mathbf{I}^T_{n_3} \otimes \mathbf{I}^T_{n_2} \otimes \mathcal{S}(\mathbf{T}^{\alpha}_{n_1})\right) \mathbf{z} \right|}{\mathbf{z}^*     \left(\mathbf{I}^T_{n_3} \otimes \mathbf{I}^T_{n_2} \otimes \tau(\mathcal{H}(\mathbf{T}_{n_1}^{\alpha})) \right) \mathbf{z} }, \, \Big(\mathbf{z}=\left(\mathbf{L}^T_{n_3} \otimes \mathbf{L}^T_{n_2} \otimes \mathbf{I}_{n_1} \right)\mathbf{w}\Big),
\end{eqnarray*}
where $\mathbf{A}_{n_i}=\mathbf{L}_{n_i}\mathbf{L}_{n_i}^T$ for $i=2, 3$ is the Cholesky factorization of the SPD matrix $\mathbf{A}_{n_i}$.

From (\ref{1D_bound}) and Lemma \ref{Her_Skew_Her}, it is easy to check that 
\begin{eqnarray*}
\frac{ \left|\mathbf{z}^*     \left(\mathbf{I}^T_{n_3} \otimes \mathbf{I}^T_{n_2} \otimes \mathcal{S}(\mathbf{T}^{\alpha}_{n_1})\right) \mathbf{z} \right|}{\mathbf{z}^*     \left(\mathbf{I}^T_{n_3} \otimes \mathbf{I}^T_{n_2} \otimes \tau(\mathcal{H}(\mathbf{T}_{n_1}^{\alpha})) \right) \mathbf{z} } \leq \frac{ \left|\mathbf{z}^*     \left(\mathbf{I}^T_{n_3} \otimes \mathbf{I}^T_{n_2} \otimes \mathcal{H}(\mathbf{T}^{\alpha}_{n_1})\right) \mathbf{z} \right|}{\mathbf{z}^*     \left(\mathbf{I}^T_{n_3} \otimes \mathbf{I}^T_{n_2} \otimes \tau(\mathcal{H}(\mathbf{T}_{n_1}^{\alpha})) \right) \mathbf{z} } \tan{\left( \frac{\alpha}{2} \pi \right)}\leq \frac{3}{2}\tan{\left( \frac{\alpha}{2} \pi \right)}.
\end{eqnarray*}
i.e.,
\begin{eqnarray}\label{skew_1}
\frac{\left|\mathbf{w}^*     \left(\mathbf{A}_{n_1^{+}} \otimes \mathcal{S}(\mathbf{T}^{\alpha}_{n_1}) \right) \mathbf{w} \right|}{\mathbf{w}^*     \left(\mathbf{A}_{n_1^{+}} \otimes \tau(\mathcal{H}(\mathbf{T}_{n_1}^{\alpha})) \right) \mathbf{w} }\leq \frac{3}{2}\tan{\left( \frac{\alpha}{2} \pi \right)}.
\end{eqnarray}
Similarly, we have 
\begin{eqnarray}\label{skew_2}
\frac{\left|\mathbf{w}^*     \left(\mathbf{A}_{n_3} \otimes \mathcal{S}(\mathbf{T}^{\beta}_{n_2}) \otimes \mathbf{A}_{n_1} \right) \mathbf{w} \right|}{\mathbf{w}^*     \left(\mathbf{A}_{n_3} \otimes \tau(\mathcal{H}(\mathbf{T}_{n_2}^{\beta})) \otimes \mathbf{A}_{n_1}\right) \mathbf{w} }\leq \frac{3}{2}\tan{\left( \frac{\beta}{2} \pi \right)},
\end{eqnarray}
and 
\begin{eqnarray}\label{skew_3}
\frac{\left|\mathbf{w}^*     \left(\mathcal{S}(\mathbf{T}^{\gamma}_{n_3}) \otimes \mathbf{A}_{n_3^{-}} \right) \mathbf{w} \right|}{\mathbf{w}^*     \left(\tau(\mathcal{H}(\mathbf{B}_{n_3}^{\gamma})) \otimes \mathbf{A}_{n_3^{-}} \right) \mathbf{w} }\leq \frac{3}{2}\tan{\left( \frac{\gamma}{2} \pi \right)}.
\end{eqnarray}

It follows from (\ref{skew_eig}), (\ref{skew_1})-(\ref{skew_3}) and Lemma \ref{ineq_lemma} that
\begin{eqnarray*}
\rho\left(\mathbf{P}^{-\frac{1}{2}} \mathcal{S}\left(  \mathbf{A}   \right) \mathbf{P}^{-\frac{1}{2}}\right) \leq \frac{3}{2} \max _{i \in 1 \wedge 3} \tan{\left( \frac{\delta_i}{2} \pi \right)}\frac{\left| k_{i,+}-k_{i,-} \right|}{ k_{i,+}+k_{i,-} }.
\end{eqnarray*}
The proof is complete.
\end{proof}
Now, it is ready to show the convergence of the GMRES method for solving the two-sided systems (\ref{two_side_system}).
\begin{theorem}\label{theorem_main}
    Let $\mathbf{A}, \mathbf{P}$ be the matrices defined in \eqref{linear_sys} and \eqref{pre_definition}, respectively. Then, the residuals of the iterates generated by applying (restarted or non-restarted) GMRES with an arbitrary initial guess to solve $\mathbf{P}^{-\frac{1}{2}}  \mathbf{A}  \mathbf{P}^{-\frac{1}{2}}\mathbf{v}=\mathbf{P}^{-\frac{1}{2}} \mathbf{b}$ satisfy
    \[
       \| \mathbf{r}_k \|_2 \leq  \omega^{k} \| \mathbf{r}_0 \|_2,
    \]
    where $\mathbf{r}_k = \mathbf{P}^{-\frac{1}{2}} \mathbf{b} - \mathbf{P}^{-\frac{1}{2}} \mathbf{A} \mathbf{P}^{-\frac{1}{2}}\mathbf{v}_k$ is the residual vector at the $k$-th GMRES iteration with $\mathbf{v}_k$ ($k \geq 1$) being the corresponding iterative solution, and $\omega$ is a constant independent of $N$ defined as follows
        $$
        \omega := \sqrt{\frac{2+4\varsigma^2}{3+4\varsigma^2}} \in \left[ \sqrt{\frac{2}{3}},1  \right) \subset (0,1),
        $$
        with $\varsigma$ defined in Lemma \ref{lemma:eigen_Skew}.
\end{theorem}
\begin{proof}
First of all, note that 
\begin{eqnarray*}
    &&\mathcal{H} \left( \mathbf{P}^{-\frac{1}{2}}  \mathbf{A}  \mathbf{P}^{-\frac{1}{2}} \right)\\
    &=&  \mathbf{P}^{-\frac{1}{2}} \mathcal{H}(\mathbf{A}) \mathbf{P}^{-\frac{1}{2}} \\
    &=&  \mathbf{P}^{-\frac{1}{2}} \left(\mathbf{A}_N +\eta_\alpha \mathbf{A}_{n_1^{+}} \otimes \mathcal{H}(\mathbf{B}_{n_1}^{\alpha})+\eta_\beta \mathbf{A}_{n_3} \otimes \mathcal{H}(\mathbf{B}_{n_2}^{\beta}) \otimes \mathbf{A}_{n_1}+\eta_\gamma \mathcal{H}(\mathbf{B}_{n_3}^{\gamma}) \otimes \mathbf{A}_{n_3^{-}}\right) \mathbf{P}^{-\frac{1}{2}}.
\end{eqnarray*}
As has been pointed out in \cite{fu2019stability} that $\mathcal{H}(\mathbf{T}_{n_i}^{\delta_i})$ is SPD. By Lemmas \ref{tridiag_eig} and \ref{spd_tau_pre}, it follows that $\mathcal{H} \left( \mathbf{P}^{-\frac{1}{2}}  \mathbf{A}  \mathbf{P}^{-\frac{1}{2}} \right)$ is positive definite. Therefore, Lemma \ref{lemma:gmres} can be used.

Since $\mathcal{H}\left(\mathbf{P}^{-1/2} \mathbf{A} \mathbf{P}^{-1/2}\right)=\mathbf{P}^{-1/2} \mathcal{H}\left( \mathbf{A} \right) \mathbf{P}^{-1/2}$ and $\mathcal{S}\left( \mathbf{P}^{-\frac{1}{2}}  \mathbf{A}  \mathbf{P}^{-\frac{1}{2}} \right)=\mathbf{P}^{-\frac{1}{2}} \mathcal{S} (\mathbf{A}) \mathbf{P}^{-\frac{1}{2}}$, 
we know by Lemmas \ref{prop:eigen_S} and \ref{lemma:eigen_Skew} that
\begin{equation*}
    \frac{3}{2}< \lambda\left( \mathcal{H}(\mathbf{P}^{-1/2} \mathbf{A} \mathbf{P}^{-1/2}) \right) < \frac{1}{2}
\end{equation*}
and
\begin{equation*}
\rho\left( \mathcal{S}\left( (\mathbf{P}^{-1/2} \mathbf{A} \mathbf{P}^{-1/2} \right)\right) \leq \varsigma,
\end{equation*}
which combined with Lemma \ref{lemma:gmres} indicates  that the residuals of the iterates generated by applying (restarted or non-restarted) GMRES method with an arbitrary initial guess to solve $\mathbf{P}^{-\frac{1}{2}}  \mathbf{A}  \mathbf{P}^{-\frac{1}{2}}\mathbf{v}=\mathbf{P}^{-\frac{1}{2}} \mathbf{b}$ satisfy
\begin{eqnarray*}
\| \mathbf{r}_k \|_2 &\leq& \left( \sqrt{1 - \left(\frac{ (\frac{1}{2})^2}{ (\frac{1}{2})(\frac{3}{2}) + \varsigma^2} \right)}     \right)^{k} \| \mathbf{r}_0 \|_2\\
&=& \left(\sqrt{\frac{2+4\varsigma^2}{3+4\varsigma^2}}  \right)^{k} \| \mathbf{r}_0 \|_2.
\end{eqnarray*}   
\end{proof}
Theorem \ref{theorem_main} establishes that for the two-sided preconditioned system (\ref{two_side_system}), the GMRES method achieves a linear convergence rate with iteration numbers independent of the matrix size. Next, we will present the relationship of the residuals when applying the GMRES method to the one-sided (\ref{one_side_system}) and two-sided preconditioned systems (\ref{two_side_system}).

Noting that $\lambda_{\min}(\mathbf{P}) \geq \lambda_{\min}(\mathbf{A}_N) \geq (\frac{1}{2})^3 = \frac{1}{8}$ from Lemma \ref{tridiag_eig}, the following theorem follows directly from the proofs in \cite{huang2024optimal,huang2025efficient,lin2024single}.
\begin{theorem}\label{residual_relationship}
Let $\hat{\mathbf{u}}_0$ be the initial guess for (\ref{two_side_system}) and $\mathbf{u}_0 := \mathbf{P}^{-1/2}\hat{\mathbf{u}}_0$ be the initial guess for (\ref{one_side_system}). Let $\mathbf{u}_j$ ($\hat{\mathbf{u}}_j$, respectively) be the $j$-th $(j\geq1)$ iteration solution derived by applying the GMRES solver to (\ref{one_side_system}) ((\ref{two_side_system}), respectively) with $\mathbf{u}_0$ ($\hat{\mathbf{u}}_0$, respectively) as their initial guess. Then,
\begin{equation*}
\left\|\mathbf{r}_j\right\|_2 \leq 2\sqrt{2}\left\|\hat{\mathbf{r}}_j\right\|_2
\end{equation*}
where $\mathbf{r}_j:=\mathbf{P}^{-1} \mathbf{b}^{m} - \mathbf{P}^{-1} \mathbf{A} \mathbf{u}_j$ ($\hat{\mathbf{r}}_j := \mathbf{P}^{-1/2} \mathbf{b}^{m} - \mathbf{P}^{-\frac{1}{2}}  \mathbf{A}  \mathbf{P}^{-\frac{1}{2}} \hat{\mathbf{u}}_j$, respectively) denotes the residual vector at the $j$-th GMRES  iteration for (\ref{one_side_system}) ((\ref{two_side_system}), respectively). 
\end{theorem}

Theorem \ref{residual_relationship} illustrates that the convergence of the GMRES method for the one-sided preconditioned systems (\ref{one_side_system}) is supported by the convergence of that for the two-sided preconditioned systems (\ref{two_side_system}), which, combined with Theorem \ref{theorem_main} implies that when GMRES method is used for solving the one-sided preconditioned systems (\ref{one_side_system}), the convergence rate is linear and the number of iteration is independent of the matrix size.

\section{Numerical experiments}\label{sec:num}
In this section, two numerical examples, covering both symmetric and non-symmetric cases in 2D and 3D problems, are tested to demonstrate the efficiency of the proposed preconditioning strategy. In the implementations, we set $n_1=n_2=n_3=n$ and adopt the MATLAB built-in functions $\mathbf{pcg}$ and $\mathbf{gmres}$ with $\mathit{maxit}=n^3$, $\mathit{restart}=20$ and $tol=10^{-9}$ for symmetric and non-symmetric case, respectively. All numerical experiments are performed on a Lenovo laptop with 16GB RAM, AMD Ryzen 5 4600U with Radeon Graphics $@$ 2.10 GHz using MATLAB R2019a.

In the tables, `PCG($\tau$)/PGMRES($\tau$)' means the CG/GMRES method with our proposed preconditioner $\mathbf{P}$, while `CG/GMRES' means those without preconditioner. For further comparison, CG/GMRES with two circulant-type preconditioners, T. Chan's circulant preconditioner \cite{chan2007introduction,fu2019finite} and Strang's circulant preconditioner \cite{chan2007introduction,fu2019finite}, are also used to solve the corresponding systems. More precisely, all Toeplitz matrices in the coefficient matrix of \eqref{original_matrix_form} are substituted with the two kinds of circulant preconditioners for preconditioning \eqref{original_matrix_form} when CG/GMRES method is used, which are denoted as `PCG(T)/PGMRES(T)' and `PCG(S)/PGMRES(S)', respectively. In addition, ``CPU'' denote the CPU time in seconds for solving the related systems;  `Iter' stands for the average iteration numbers of different methods. When the CPU time exceeds 3000 seconds, we stop the iteration and represent the results in the tables as ``$\dagger$''.


\begin{example}\label{ex_1}
Consider problem (\ref{RLFDEs}) in two dimensions with $\Omega=(0,1)^2$, $T=1$. The exact solution is
\begin{equation*}
u(x, y, t)=  4e^{t}  x^{2}(1-x)^{2} y^{2}(1-y)^{2},
\end{equation*}
and the source term
{\small
\begin{align} \nonumber
g(x, y, t)=& 4e^{t}\Big\{ x^{2} (1-x)^{2} y^{2}(1-y)^{2}\\ \nonumber
& - y^{2}(1-y)^{2} \sum_{k=0}^{2}(-1)^{2-k}  C_{2}^{k} \frac{\Gamma(5-k)}{\Gamma(3-k+\alpha)} \Big[k_{1,+}x^{2-k+\alpha}+k_{1,-}(1-x)^{2-k+\alpha}\Big]\\ \nonumber
&-x^{2}(1-x)^{2} \sum_{k=0}^{2}(-1)^{2-k} C_{2}^{k} \frac{\Gamma(5-k)}{\Gamma(3-k+\beta)} \Big[k_{2,+}y^{2-k+\beta}+k_{2,-}(1-y)^{2-k+\beta}\Big]\Big\}.
\end{align}
}
\end{example}
We first consider the symmetric case for Example \ref{ex_1} with $k_{1,+}=k_{1,-}=k_{2,+}=k_{2,-}=5$. CPU time and iteration numbers with different methods are summarized in Table \ref{table1}. As shown in the table, for different combinations of fractional orders, both preconditioners can effectively reduce the CPU time for convergence by decreasing the number of iterations. For fixed fractional orders, the iteration numbers required by PCG($\tau$) are almost constant as the matrix size increases, agreeing to the theoretical analysis. In contrast, the iteration counts of PCG(T) and PCG(S) are both greater than those of PCG($\tau$) and exhibit a significant increase, especially when the fractional orders approach 0 or in anisotropic cases (e.g., $(\alpha,\beta)=(0.1, 0.9)$). 

For the non-symmetric case, $k_{1,+}=19$, $k_{1,-}=21$, $k_{2,+}=21$, $k_{2,-}=23$ are considered. The numerical results are displayed in Table \ref{table2}. Similar to the symmetric case, PGMRES($\tau$) in the non-symmetric case exhibits the most favorable performance among the tested methods in terms of both computational efficiency and robustness. It consistently requires the smallest and constant number of iterations and the lowest CPU time across all tested fractional orders.


\begin{table}[t]%
\centering
\tabcolsep=2.6pt
\renewcommand{\arraystretch}{0.8}
\caption{The comparisons between the CG and PCG methods with the aforementioned preconditioners for solving Example \ref{ex_1} with $k_{1,+}=k_{1,-}=k_{2,+}=k_{2,-}=5$ at $T=1$.} 
\label{table1}
\begin{tabular}{ccccccccccc}
\toprule
\multirow{2}{*}{$(\alpha,\beta)$} & \multirow{2}{*}{$M$}  & \multirow{2}{*}{$n+1$}  &  \multicolumn{2}{c}{CG} &  \multicolumn{2}{c}{PCG(T)} & \multicolumn{2}{c}{PCG(S)}  & \multicolumn{2}{c}{PCG($\tau$)} \\  \cmidrule(r){4-5} \cmidrule(r){6-7} \cmidrule(r){8-9}\cmidrule(r){10-11}
     &    &  &CPU(s)  &Iter &CPU(s)  &Iter &CPU(s)  &Iter &CPU(s)  &Iter  \\
\midrule
\multirow{4}{*}{(0.1, 0.2)}
         &$2^{3}$   &$2^{6}$   &1.12         &103.00    &0.22     &29.88    &0.16     &19.75    &0.13   &6.00  \\
         &$2^{4}$   &$2^{7}$   &7.78         &200.00    &2.20     &43.88    &1.35     &27.19    &0.57   &7.00  \\
         &$2^{5}$   &$2^{8}$   &103.25       &383.00    &20.27    &64.97    &10.88    &33.41    &4.66   &7.00   \\
         &$2^{6}$   &$2^{9}$   &2355.98      &728.00    &380.77   &99.75    &172.78   &44.33    &48.52  &7.00   \\
\midrule
\multirow{4}{*}{(0.4, 0.5)}
         &$2^{3}$   &$2^{6}$   &0.83         &73.00     &0.14     &22.00    &0.11     &17.38    &0.10   &7.00  \\
         &$2^{4}$   &$2^{7}$   &5.18         &126.00    &1.52     &29.44    &1.09     &20.94    &0.59   &8.00  \\
         &$2^{5}$   &$2^{8}$   &62.06        &210.00    &12.34    &38.41    &8.10     &24.38    &5.16   &8.00   \\
         &$2^{6}$   &$2^{9}$   &1160.12      &345.00    &190.34   &48.98    &117.38   &29.94    &53.99  &8.00   \\
\midrule
\multirow{4}{*}{(0.8, 0.9)}
         &$2^{3}$   &$2^{6}$   &0.22         &42.00    &0.10    &15.00    &0.09    &14.00    &0.11   &8.00  \\
         &$2^{4}$   &$2^{7}$   &2.13         &55.00    &0.85    &16.00    &0.75    &14.00    &0.59   &8.00  \\
         &$2^{5}$   &$2^{8}$   &19.28        &68.00    &5.59    &16.00    &5.35    &15.00    &5.18   &8.00   \\
         &$2^{6}$   &$2^{9}$   &274.09       &78.00    &64.39   &15.00    &63.16   &15.00    &54.07  &8.00   \\  
\midrule         
 \multirow{4}{*}{(0.1, 0.9)}
         &$2^{3}$   &$2^{6}$   &1.12         &246.88    &0.28     &47.88    &0.18     &30.25    &0.09   &6.00  \\
         &$2^{4}$   &$2^{7}$   &17.39        &473.94    &3.20     &67.13    &1.74     &35.69    &0.52   &7.00  \\
         &$2^{5}$   &$2^{8}$   &230.29       &843.00    &27.98    &90.72    &13.09    &40.97    &4.59   &7.00   \\
         &$2^{6}$   &$2^{9}$   &$>3000$      &$\dagger$ &430.85   &114.83   &189.47   &48.98    &48.48  &7.00   \\         
\bottomrule
\end{tabular}
\end{table}

\begin{table}[t]%
\centering
\tabcolsep=2.6pt
\renewcommand{\arraystretch}{0.8}
\caption{The comparisons between the GMRES and PGMRES methods with the aforementioned preconditioners for solving Example \ref{ex_1} with $k_{1,+}=19$, $k_{1,-}=21$, $k_{2,+}=21$, $k_{2,-}=23$ at $T=1$.} 
\label{table2}
\begin{tabular}{ccccccccccc}
\toprule
\multirow{2}{*}{$(\alpha,\beta)$} & \multirow{2}{*}{$M$}  & \multirow{2}{*}{$n+1$}  &  \multicolumn{2}{c}{GMRES} &  \multicolumn{2}{c}{PGMRES(T)} & \multicolumn{2}{c}{PGMRES(S)}  & \multicolumn{2}{c}{PGMRES($\tau$)} \\  \cmidrule(r){4-5} \cmidrule(r){6-7} \cmidrule(r){8-9}\cmidrule(r){10-11}
     &    &  &CPU(s)  &Iter &CPU(s)  &Iter &CPU(s)  &Iter &CPU(s)  &Iter  \\
\midrule
\multirow{4}{*}{(0.1, 0.2)}
         &$2^{3}$   &$2^{6}$   &3.62         &393.00     &0.41     &42.75    &0.21     &23.00    &0.12   &6.00  \\
         &$2^{4}$   &$2^{7}$   &61.46        &1220.00    &5.04     &75.00    &2.05     &31.00    &0.58   &6.00  \\
         &$2^{5}$   &$2^{8}$   &1414.67      &4059.50    &55.52    &143.00   &14.55    &37.00    &5.15   &6.00   \\
         &$2^{6}$   &$2^{9}$   &$>3000$      &$\dagger$  &1209.13  &266.00   &233.65   &51.00    &50.60  &6.00   \\ 
\midrule
\multirow{4}{*}{(0.4, 0.5)}
         &$2^{3}$   &$2^{6}$   &1.94         &172.00     &0.27     &30.00    &0.18     &20.00    &0.14   &8.00  \\
         &$2^{4}$   &$2^{7}$   &21.66        &400.00     &3.07     &45.00    &1.92     &28.00    &0.72   &8.00  \\
         &$2^{5}$   &$2^{8}$   &341.29       &964.50     &24.08    &62.00    &12.49    &32.00    &6.25   &8.00   \\
         &$2^{6}$   &$2^{9}$   &$>3000$      &$\dagger$  &402.26   &87.00    &171.36   &37.00    &62.69  &8.00   \\ 
\midrule
\multirow{4}{*}{(0.8, 0.9)}
         &$2^{3}$   &$2^{6}$   &0.57       &81.00     &0.18     &20.00   &0.16     &18.00    &0.18   &11.00  \\
         &$2^{4}$   &$2^{7}$   &6.73       &121.00    &1.59     &23.00   &1.27     &19.00    &0.95   &11.00  \\
         &$2^{5}$   &$2^{8}$   &62.86      &176.00    &9.87     &24.00   &9.16     &22.00    &8.19   &11.00   \\
         &$2^{6}$   &$2^{9}$   &907.79     &228.00    &117.01   &25.00   &111.24   &23.00    &81.14  &11.00   \\  
\midrule         
 \multirow{4}{*}{(0.1, 0.9)}
         &$2^{3}$   &$2^{6}$   &3.23         &471.00     &0.47     &58.00    &0.45     &56.00    &0.14   &8.00  \\
         &$2^{4}$   &$2^{7}$   &106.87       &1969.00    &5.66     &88.00    &4.44     &70.00    &0.80   &9.00  \\
         &$2^{5}$   &$2^{8}$   &2308.52      &6589.56    &53.12    &139.00   &38.15    &100.00   &6.89   &9.00   \\
         &$2^{6}$   &$2^{9}$   &$>3000$      &$\dagger$  &944.60   &210.00   &566.72   &122.94   &68.92  &9.00   \\           
\bottomrule
\end{tabular}
\end{table}

\begin{example} \label{ex_2}
Consider problem (\ref{RLFDEs}) in three dimensions with $\Omega=(0,1)^3$, $T=1$. The exact solution is
\begin{equation*}
u(x, y, t)= \sin(t+1) x^{2}(1-x)^{2} y^{2}(1-y)^{2} z^{2}(1-z)^{2},
\end{equation*}
and the source term
{\small
\begin{align}   \nonumber
g(x, y, t)=& \cos(t+1) x^{2}(1-x)^{2} y^{2}(1-y)^{2} z^{2}(1-z)^{2} - \sin(t+1) \times\\  \nonumber
& \Big\{y^{2}(1-y)^{2}z^{2}(1-z)^{2} \sum_{k=0}^{2}(-1)^{2-k}  C_{2}^{k} \frac{\Gamma(5-k)}{\Gamma(3-k+\alpha)} \Big[k_{1,+}x^{2-k+\alpha}+k_{1,-}(1-x)^{2-k+\alpha}\Big]\\ \nonumber
&+x^{2}(1-x)^{2} z^{2}(1-z)^{2}\sum_{k=0}^{2}(-1)^{2-k} C_{2}^{k} \frac{\Gamma(5-k)}{\Gamma(3-k+\beta)} \Big[k_{2,+}y^{2-k+\beta}+k_{2,-}(1-y)^{2-k+\beta}\Big]\\ \nonumber
&+ x^{2}(1-x)^{2} y^{2}(1-y)^{2} \sum_{k=0}^{2}(-1)^{2-k} C_{2}^{k} \frac{\Gamma(5-k)}{\Gamma(3-k+\gamma)} \Big[k_{3,+}z^{2-k+\gamma}+k_{3,-}(1-z)^{2-k+\gamma}\Big]\Big\}.
\end{align}
}
\end{example}   

Similarly, for the 3D problem displayed in Example \ref{ex_2}, both symmetric and non-symmetric cases are discussed with $k_{1,+}=k_{1,-}=k_{2,+}=k_{2,-}=k_{3,+}=k_{3,-}=5$ and $k_{1,+}=19$, $k_{1,-}=21$, $k_{2,+}=21$, $k_{2,-}=23$, $k_{3,+}=23$, $k_{3,-}=25$, respectively. Numerical results for these two cases are reported in Tables \ref{table3} and \ref{table4}. 
Consistent with the 2D case, PCG($\tau$)/PGMRES($\tau$) for the 3D case also outperforms other methods in terms of both iteration counts and CPU time across varying fractional orders. Compared to PCG(T)/PGMRES(T) and PCG(S)/PGMRES(S) whose iteration counts grow significantly, PCG($\tau$)/PGMRES($\tau$) maintains the fastest convergence with nearly constant iterations and lowest CPU time.

In conclusion, the consistent superiority of PCG($\tau$)/PGMRES($\tau$) in both symmetric and non-symmetric cases validates the theoretical analysis and demonstrates the efficiency and robustness of the proposed preconditioner.

\begin{table}[t]%
\centering
\tabcolsep=4.2pt
\renewcommand{\arraystretch}{0.5}
\caption{The comparisons between the CG and PCG methods with the aforementioned preconditioners for solving Example \ref{ex_2} with $k_{1,+}=k_{1,-}=k_{2,+}=k_{2,-}=k_{3,+}=k_{3,-}=5$ at $T=1$.} 
\label{table3}
\begin{tabular}{ccccccccccc}
\toprule
\multirow{2}{*}{$(\alpha_1,\alpha_2,\alpha_3)$} & \multirow{2}{*}{$M$}  & \multirow{2}{*}{$n+1$}  &  \multicolumn{2}{c}{CG} & \multicolumn{2}{c}{PCG(T)} & \multicolumn{2}{c}{PCG(S)} & \multicolumn{2}{c}{PCG($\tau$)} \\  \cmidrule(r){4-5} \cmidrule(r){6-7} \cmidrule(r){8-9} \cmidrule(r){10-11}
     &    &  &CPU  &Iter &CPU  &Iter &CPU  &Iter &CPU  &Iter \\
\midrule
\multirow{4}{*}{(0.1, 0.2, 0.3)}
         &$2^{2}$   &$2^{3}$   &0.03     &17.00       &0.02    &12.00    &0.03    &13.00    &0.05   &5.00  \\
         &$2^{3}$   &$2^{4}$   &0.33     &34.00       &0.24    &18.00    &0.20    &17.00    &0.22   &6.00  \\
         &$2^{4}$   &$2^{5}$   &6.53     &66.00       &3.17    &26.00    &2.70    &21.44    &1.37   &6.00  \\
         &$2^{5}$   &$2^{6}$   &351.13   &126.00      &120.77  &38.00    &85.84   &26.97    &39.86  &7.00   \\
\midrule
\multirow{4}{*}{(0.4, 0.5, 0.6)}
         &$2^{2}$   &$2^{3}$   &0.02     &16.00       &0.02    &10.00    &0.02    &11.00    &0.04   &6.00  \\
         &$2^{3}$   &$2^{4}$   &0.22     &28.00       &0.14    &15.00    &0.14    &14.00    &0.17   &7.00  \\
         &$2^{4}$   &$2^{5}$   &4.66     &47.00       &2.52    &20.00    &2.18    &17.00    &1.55   &8.00  \\
         &$2^{5}$   &$2^{6}$   &223.27   &79.00       &82.58   &26.00    &65.39   &20.00    &44.29  &8.00   \\
\midrule
\multirow{4}{*}{(0.7, 0.8, 0.9)}
         &$2^{2}$   &$2^{3}$   &0.02     &16.00      &0.02    &9.00     &0.02    &11.00    &0.04   &7.00  \\
         &$2^{3}$   &$2^{4}$   &0.19     &24.00      &0.12    &12.00    &0.13    &13.00    &0.19   &8.00  \\
         &$2^{4}$   &$2^{5}$   &3.40     &34.00      &1.83    &14.00    &1.89    &14.00    &1.60   &8.00  \\
         &$2^{5}$   &$2^{6}$   &127.32   &44.00      &54.31   &16.00    &51.00   &15.00    &43.91  &8.00   \\  
\midrule
\multirow{4}{*}{(0.1, 0.5, 0.9)}
         &$2^{2}$   &$2^{3}$   &0.05     &32.75       &0.03    &15.00    &0.02    &17.00    &0.04   &6.00  \\
         &$2^{3}$   &$2^{4}$   &0.48     &65.25       &0.22    &23.00    &0.23    &24.88    &0.15   &6.00  \\
         &$2^{4}$   &$2^{5}$   &11.25    &118.75      &4.17    &35.00    &3.84    &31.75    &1.38   &7.00  \\
         &$2^{5}$   &$2^{6}$   &602.17   &218.00      &161.34  &52.00    &114.31  &36.97    &39.62  &7.00   \\ 
\bottomrule
\end{tabular}
\end{table}

\begin{table}[t]%
\centering
\tabcolsep=3.2pt
\renewcommand{\arraystretch}{0.5}
\caption{The comparisons between the GMRES and PGMRES methods with the aforementioned preconditioners for solving Example \ref{ex_2} with $k_{1,+}=19$, $k_{1,-}=21$, $k_{2,+}=21$, $k_{2,-}=23$, $k_{3,+}=23$, $k_{3,-}=25$ at $T=1$.} 
\label{table4}
\begin{tabular}{ccccccccccc}
\toprule
\multirow{2}{*}{$(\alpha_1,\alpha_2,\alpha_3)$} & \multirow{2}{*}{$M$}  & \multirow{2}{*}{$n+1$}  &  \multicolumn{2}{c}{GMRES} & \multicolumn{2}{c}{PGMRES(T)} & \multicolumn{2}{c}{PGMRES(S)} & \multicolumn{2}{c}{PGMRES($\tau$)} \\  \cmidrule(r){4-5} \cmidrule(r){6-7} \cmidrule(r){8-9} \cmidrule(r){10-11}
     &    &  &CPU  &Iter &CPU  &Iter &CPU  &Iter &CPU  &Iter \\
\midrule
\multirow{4}{*}{(0.1, 0.2, 0.3)}
         &$2^{2}$   &$2^{3}$   &0.03     &21.00       &0.02    &13.00    &0.02    &14.00    &0.05   &6.00  \\
         &$2^{3}$   &$2^{4}$   &0.54     &50.00       &0.23    &20.00    &0.22    &18.00    &0.19   &6.00  \\
         &$2^{4}$   &$2^{5}$   &12.47    &103.00      &4.51    &32.00    &3.89    &27.00    &1.42   &6.00  \\
         &$2^{5}$   &$2^{6}$   &974.58   &308.00      &178.62  &51.00    &116.58  &33.00    &42.31  &7.00   \\
\midrule
\multirow{4}{*}{(0.4, 0.5, 0.6)}
         &$2^{2}$   &$2^{3}$   &0.03     &19.00       &0.02    &12.00    &0.02    &13.00    &0.05   &7.00  \\
         &$2^{3}$   &$2^{4}$   &0.37     &37.00       &0.21    &18.00    &0.20    &17.00    &0.22   &8.00  \\
         &$2^{4}$   &$2^{5}$   &8.37     &69.00       &3.63    &25.00    &3.28    &22.00    &1.78   &8.00  \\
         &$2^{5}$   &$2^{6}$   &435.91   &138.00      &118.73  &33.00    &98.78   &28.00    &48.96  &8.00   \\
\midrule
\multirow{4}{*}{(0.7, 0.8, 0.9)}
         &$2^{2}$   &$2^{3}$   &0.03     &19.00      &0.02    &11.50    &0.02    &13.00    &0.06   &9.00  \\
         &$2^{3}$   &$2^{4}$   &0.32     &33.00      &0.17    &15.00    &0.19    &16.00    &0.27   &10.00  \\
         &$2^{4}$   &$2^{5}$   &6.05     &50.00      &2.71    &19.00    &2.74    &19.00    &2.11   &10.00  \\
         &$2^{5}$   &$2^{6}$   &236.23   &75.00      &82.43   &22.00    &82.13   &22.00    &58.24  &10.00   \\  
\midrule
\multirow{4}{*}{(0.1, 0.5, 0.9)}
         &$2^{2}$   &$2^{3}$   &0.06     &40.50       &0.03    &17.00    &0.04    &22.00    &0.05   &7.00  \\
         &$2^{3}$   &$2^{4}$   &0.71     &75.00       &0.30    &26.00    &0.38    &33.50    &0.22   &7.00  \\
         &$2^{4}$   &$2^{5}$   &20.43    &171.00      &5.65    &40.00    &6.87    &47.00    &1.73   &8.00  \\
         &$2^{5}$   &$2^{6}$   &1416.09  &447.00      &217.46  &62.00    &209.19  &59.00    &49.56  &8.00   \\ 
\bottomrule
\end{tabular}
\end{table}

\section{Conclusions}\label{sec:con}
In this paper, we propose sine transform-based preconditioning techniques for the CN-FV scheme of 3D conservative SFDEs. Theoretically, we prove that the proposed preconditioner enables the PCG and PGMRES methods to achieve linear convergence rates for symmetric and non-symmetric cases, respectively.
Moreover, our preconditioning techniques not only extend the applicability of the sine transform-based preconditioning techniques from the FD method to the FV method, but also significantly enhance the performance of the circulant preconditioners proposed in \cite{fu2019stability,fu2019finite}, as demonstrated by numerical results. In future work, we will focus on solving high-dimensional conservative SFDEs with variable coefficients.

\section*{Acknowledgments}
The work of Wei Qu was supported by the research grant 2024KTSCX069 from the Characteristic Innovation Projects of Ordinary Colleges and Universities in Guangdong Province. The work of Siu-Long Lei was supported by the research grants MYRG-GRG2024-00237-FST-UMDF and MYRG-GRG2023-00181-FST-UMDF from University of Macau. The work of Sean Y. Hon was supported in part by NSFC under grant 12401544, the Guangdong and Hong Kong Universities ``1+1+1'' Joint Research Collaboration Scheme (project no. 2025A0505000014), and a start-up grant from the Croucher Foundation.

\bibliographystyle{elsarticle-num-names}
\bibliography{ref}
\end{document}